%% file: draft.tex
\documentclass{amsart}

\usepackage[foot]{amsaddr}

\usepackage{subfig}
\usepackage{graphicx}
\usepackage{amssymb,amsfonts,amsrefs}
\usepackage{hyperref}
\usepackage{tikz-cd}
\usepackage{enumitem}
\usepackage{array}
\usepackage{changes}
\usepackage{ulem}

\usepackage[all,cmtip]{xy}

\usepackage{bbm}
\usepackage{tikz-cd}

\usepackage{tikz}
\usetikzlibrary{matrix,arrows}

\usepackage{mathtools}
\usetikzlibrary{calc}

\makeatletter
\newtheorem*{rep@theorem}{\rep@title}
\newcommand{\newreptheorem}[2]{%
\newenvironment{rep#1}[1]{%
 \def\rep@title{#2 \ref{##1}}%
 \begin{rep@theorem}}%
 {\end{rep@theorem}}}
\makeatother

\newtheorem{thm}{Theorem}[section]
\newreptheorem{thm}{Theorem}
\newreptheorem{cor}{Corollary}
\newreptheorem{conj}{Conjecture}
\newtheorem{cor}[thm]{Corollary}
\newtheorem*{cor*}{Corollary}
\newtheorem{prop}[thm]{Proposition}
\newtheorem*{prop*}{Proposition}
\newtheorem*{thm*}{Theorem}
\newtheorem{lem}[thm]{Lemma}

\newtheorem{mainthm}{Theorem}

\theoremstyle{definition}
\newtheorem{defn}[thm]{Definition}
\newtheorem*{defn*}{Definition}

\newtheorem{ex}[thm]{Example}

\theoremstyle{remark}
\newtheorem{rem}[thm]{Remark}
\newtheorem*{rem*}{Remark}
\newtheorem{rems}[thm]{Remarks}

\DeclareMathOperator{\Hom}{Hom}

\DeclareMathOperator{\Br}{Br}
\DeclareMathOperator{\im}{im}
\DeclareMathOperator{\Gal}{Gal}
\DeclareMathOperator{\res}{res}
\DeclareMathOperator{\Aut}{Aut}

\DeclareMathOperator{\cores}{cor}

\newcommand{\HNN}{{\rm HNN}}

\newcommand{\KK}{\mathbb{K}}
\newcommand{\LL}{\mathbb{L}}
\newcommand{\RR}{\mathbb{R}}

\newcommand{\ZZ}{\mathbb{Z}}
\newcommand{\NN}{\mathbb{N}}
\newcommand{\QQ}{\mathbb{Q}}
\newcommand{\FF}{\mathbb{F}}

\newcommand{\calR}{\mathcal{R}}

\newcommand{\calG}{\mathcal{G}}

\newcommand{\rmH}{\mathrm H}%for cohomology groups
\newcommand{\bfH}{\mathbf{H}^\bullet}%for the whole cohomology algebra
\newcommand{\set}[2]{\left\lbrace{#1}\ \big\vert\ {#2}\right\rbrace}
\newcommand{\pres}[2]{\left\langle{#1}\, \big\vert\, {#2}\right\rangle}

\newcommand{\K}{\KK}
\newcommand{\Z}{\ZZ}
\newcommand{\F}{\FF}
\newcommand{\pth}{{p^{\rm th}}}

\definecolor{julian}{RGB}{204,102,0}

\usepackage[bb=ncmbbr]{mathalpha}

\title[Cohomological Kaplansky radical for profinite groups]{A cohomological translation of the Kaplansky radical for profinite groups}

\author{S. Blumer}
\email{\href{mailto:simone.blumer@unimib.it}{simone.blumer@unimib.it}}
\author{J. Feuerpfeil}
\email{\href{mailto:j.feuerpfeil@campus.unimib.it}{j.feuerpfeil@campus.unimib.it}}
\author{L. C. Lopes}
\email{\href{mailto:lucasclopes@ufmg.br}{lucasclopes@ufmg.br}}
\author{C. Quadrelli}
\email{\href{mailto:claudio.quadrelli@uninsubria.it}{claudio.quadrelli@uninsubria.it}}
\address[SB,JF]{Dipartimento di Matematica e Applicazioni, Università di Milano-Bicocca, Milan, Italy}

\address[JF]{Femto-St, Université Marie et Louis Pasteur, Besan\c{c}on, France}

\address[LL]{Universidade Federal de Minas Gerais, Departamento de Matemática, Belo Horizonte, Brazil}

\address[CQ]{Department of Science \& High-Tech, University of Insubria, Como, Italy}

\begin{document}
% ABSTRACT
\begin{abstract} 
The Kaplansky radical of a field consists of the nonzero elements represented by every norm quadratic form in two variables. D.~Kijima and M.~Nishi conjectured that, for quadratic extensions, the Kaplansky radicals are related by the norm map in a manner analogous to Hilbert's Theorem~90. Although this H-conjecture was disproved by K.J.~Becher and D.B.~Leep, it is known to hold for several important classes of fields.

We introduce a cohomological analogue of the Kaplansky radical for arbitrary profinite groups and primes $p$, defined as the orthogonal of ${\rm H}^1(G,\mathbb{F}_p)$ with respect to the cup product with itself. For absolute Galois groups, this recovers the classical Kaplansky radical when $p=2$ and the $p-$radical of Dario--Engler for arbitrary p. We also formulate a group-theoretic analogue of the H-conjecture, proving that, for fields, it is equivalent to the original conjectural property and depends only on the maximal pro-$2$ quotient of the absolute Galois group.

We establish this property for broad classes of fields, including local and global fields, rational function fields, and all fields whose maximal pro-$p$ Galois group is of elementary type. Beyond its arithmetic origins, we investigate the property for general pro-$p$ groups, proving its stability under several natural group-theoretic constructions and obtaining new examples, including generalized right-angled Artin pro-$p$ groups and fundamental pro-$p$ groups of suitable graphs of groups, many of which cannot occur as maximal pro-$p$ Galois groups.



\end{abstract}

\maketitle

%\tableofcontents

\input{contentsV2}

\section*{Acknowledgments}
    All the authors would like to thank Thomas Weigel for many helpful and inspiring discussions and exchanges; and Ronie Peterson Dario and Karim Johannes Becher for for their valuable feedback and comments about this manuscript.
    JF and LL are grateful to Pavel Zalesski{\u{\i}} for his guiding examples and insights on the fourth section of the article. JF also thanks Christian Maire and Oussama Hamza for interesting exchanges on this subject.

    SB, JF and CQ are members of Gruppo Nazionale per le Strutture Algebriche, Geometriche e le loro Applicazioni (GNSAGA), which is part of the Istituto Nazionale di Alta Matematica (INdAM).
    JF is part of the \lq\lq INdAM - GNSAGA Project\rq\rq\ CUP E53C25002010001. LL was financed in part by the Conselho Nacional de Desenvolvimento Científico e Tecnológico – CNPq, Processo/Grant nº 152746/2025-0.
\bibliography{literature}
\newpage

\end{document}

%% file: contentsV2.tex
\section*{Introduction}
\label{sec:intro}
The \emph{Kaplansky radical} of a field $\KK$ of characteristic not $2$ is an important invariant in the theory of quadratic forms. It was first introduced by I.~Kaplansky in~\cite{Kaplansky} and got its current name in C.M.~Cordes' article~\cite{Cordes1975}. It is defined as
\begin{align*}
    \calR_{\rm K}(\KK)\coloneq \set{a\in \KK^\times}{D_\KK\langle 1,-a\rangle=\KK^\times}
\end{align*}
where $D_\KK\langle 1,-a\rangle$ is the set of non-zero elements of $\KK$ represented by the binary quadratic form $X^2-aY^2$ (see, e.g., \cite[Chap. XII, \S 6]{Lam2005}). 

In the '80s, in a series of papers, D.~Kijima and M.~Nishi investigated the norm maps between Kaplansky radicals in quadratic field extensions in analogy to the classical Hilbert~90 theorem (see \cite{KijimaNishi1981,KijimaNishi1983,KijimaNishi1985}). They conjectured that for each $a\in \KK^\times$ one has
\begin{align}
\label{eq:KN radical equality}
    N_{\KK(\sqrt{a})/\KK}^{-1}({\calR}_{\mathrm{K}}(\KK))={\calR}_{\mathrm{K}}(\KK(\sqrt{a}))\cdot\KK^\times .
\end{align}
This conjecture became known as \lq\lq H-conjecture\rq\rq\ and was disproven about 30 years later by K.J.~Becher and D.B.~Leep in~ \cite{BecherLeep2014} by providing explicit constructions of fields where this property fails. In case the equality in \eqref{eq:KN radical equality} holds, they say that $\KK(\sqrt{a})/\KK$ is \emph{radical exact}.

{ In \cite{DarioEngler}, R.P.~Dario and A.J.~Engler proposed a generalization of the Kaplansky radical for arbitrary primes $p$.
In this paper we translate the Kaplansky radical and Dario--Engler's {\it $p$-radical} in cohomological terms, so that it may be applied to arbitrary profinite groups and primes $p$}; and we investigate the analogous of the H-conjecture for this cohomological variation of the radical.

Given a profinite group $G$ and a prime number $p$, consider the field $\FF_p$ with $p$ elements as a trivial $G$-module. We define the $\FF_p$\emph{-cup radical} $\calR_p(G)$ of $G$ to be the orthogonal of the cohomology group $\rmH^1(G,\FF_p)$ with respect to the cup product pairing. That is,
\[
 \calR_p(G)=\rmH^1(G,\FF_p)^\perp=\set{\alpha\in\rmH^1(G,\FF_p)}{\alpha\smallsmile\beta=0\:\text{for all }\beta\in\rmH^1(G,\FF_p)}.
\]
It is not difficult to see that if $U$ is an open subgroup of $G$, then corestriction on cohomology induces a map $c_{U,G}:\calR_p(U)\to \calR_p(G)$ (see Lemma~\ref{lem:maps between radicals} \ref{it:lem maps between radicals cores}).

We say that $G$ \emph{has the} $p$-\emph{Kijima--Nishi property} if $c_{U,G}$ is surjective for every normal open subgroup $U$ of index $p$.
The $\FF_p$-cup radical and the $p$-Kijima--Nishi property depend only on the maximal pro-$p$ quotient $G(p)$ of $G$, as shown in Proposition~\ref{prop:passing to the pro-p quotient}. This is the reason we mainly focus on pro-$p$ groups throughout the paper.

For a field $\KK$ we fix a separable closure $\KK^{\rm sep}$ and denote by $G_\KK=\Gal(\KK^{\rm sep}\!/\KK)$ its absolute Galois group. 
Also, for an arbitrary prime $p$ and $\KK$ a field containing a primitive $p^{\mathrm{th}}$-root of unity, let $\calR_p(\KK)$ denote the $p$-radical of $\KK$, as defined in \cite{DarioEngler}.
Then the following proposition yields a dictionary between the arithmetic situation and our group theoretic definitions:

\begin{prop*}[Lem.~\ref{lemma:RelationBetweenRadicals} and Prop.~\ref{prop:equivalence}]
    Let {$p$ be a prime number, and let} $\KK$ be a field containing a primitive $p^{\mathrm{th}}$-root of unity.
    Then we have { $\calR_{p}(\KK)/\KK^{\times p}\cong \calR_p(G_\KK)$. In particular, for $p=2$ we have} $\calR_{\rm K}(\KK)/\KK^{\times 2}\cong \calR_2(G_\KK)$, and $\KK$ satisfies the H-conjecture if and only if $G_\KK$ has the $2$-Kijima--Nishi property.
\end{prop*}
This shows that $\calR_2(G_\KK)$ may be seen as the ``group-cohomological {\sl alter ego}'' of the Kaplansky radical, and that the property whether $\KK$ satisfies the H-conjecture depends only on its absolute Galois group --- in the spirit of Anabelian Geometry. 

Our first objective is to study the $p$-Kijima--Nishi property for various fields. Following the arguments given by K.J.~Becher and D.B.~Leep in \cite{BecherLeep2014} to construct the counterexample to the H-conjecture, we prove the following.

\begin{mainthm}
    Let $p$ be an arbitrary prime. For every field $\KK$ of characteristic $\neq p$ there exists an extension $\KK'/\KK$ such that $G_{\KK'}$ does not have the $p$-Kijima--Nishi property.
\end{mainthm}

The construction heavily relies on the construction of highly transcendental extensions to split sufficiently many central simple algebras. 
Thus, it still makes sense to ask whether the absolute Galois group of a field satisfies the $p$-Kijima--Nishi property, when one restricts to fields and absolute Galois groups that are, in some sense, \lq\lq small\rq\rq.

Indeed we show that several such fields have the $p$-Kijima--Nishi property. The main results are collected in the next theorem. They can be found in the article as Theorem~\ref{thm:Gloabal field trivial radical}, Proposition~\ref{prop:Radical of discretely valued field}, Corollary~\ref{cor: function field in n variables has trivial radical}, and Theorem~\ref{thm:ETC}.
\begin{mainthm}
\label{mainthm:List of fields}
    Let $\KK$ be a field containing a primitive $\pth$-root of unity. The absolute Galois group $G_{\KK}$ satisfies the $p$-Kijima--Nishi property in the following cases:
    \begin{enumerate}[label=(\roman*)]
        \item  $\KK$ is a global or local field;
        \item  $\KK$ is a complete discretely valued field whose residue field contains a primitive $\pth$-root of unity;
        \item $\KK$ is a purely transcendental extension of another field containing a primitive $\pth$-root of unity;
        \item \label{it:Main Theorem Fields ETC} the maximal pro-$p$ quotient $G_\KK(p)$ of $G_{\KK}$ is of elementary type (see Definition~\ref{defn:ET});
        \item \label{it:Main Theorem Fields Pythagorean}$\KK$ is formally real Pythagorean of finite type and $p=2$.
    \end{enumerate}
\end{mainthm}
\begin{rem*}
    We always assume that $\KK$ contains a primitive $\pth$-root of unity to identify $\rmH^1(G_\KK,\mu_p)$ and $\rmH^1(G_\KK,\FF_p)$. If $\KK$ does not contain a primitive $\pth$-root of unity, the situation becomes more subtle and there are multiple possible definitions of the radical. We discuss one such possibility in Section~\ref{sec:No roots of unity} and give some results hinting that this could be the right definition.
\end{rem*}
We discuss case \ref{it:Main Theorem Fields ETC} in more detail, which is the content of Theorem~\ref{thm:ETC}. Groups of elementary type are pro-$p$ groups constructed from $\ZZ_p$ and Demu\v{s}kin groups, using certain semidirect products with free abelian pro-$p$ groups, and free pro-$p$ products. The semidirect product corresponds on the level of fields to passing from $\KK$ to $\KK(\!(t)\!)$. For a precise definition of this class of pro-$p$ groups, introduced by I.~Efrat in \cite{ETCsmall}, we refer to Definition~\ref{defn:ET}.

The \emph{Elementary Type Conjecture} on maximal pro-$p$ Galois groups --- due to I.~Efrat, see \cite{ETC1,ETC2} --- asserts that if $\KK$ contains a primitive $\pth$-root of unity and $\KK^\times/\KK^{\times p}$ is finite, then $G_\KK(p)$ is of elementary type. It is known to hold for several fields, including formally real Pythagorean fields for $p=2$, PAC fields and fields of transcendence degree $1$ over a PAC field (see \cite{FriedJarden2023} for a definition of PAC field).

The proof of case \ref{it:Main Theorem Fields ETC} of Theorem~\ref{mainthm:List of fields} is heavily inspired by the work of R.P.~Dario (see \cite{Dario}), who considered the situation for $p=2$. The main idea is to show that the $p$-Kijima--Nishi property is satisfied for the elementary building blocks of groups of elementary type and is stable under the \lq\lq gluing operations\rq\rq. 
We underline that Dario’s approach relies on the theory of quadratic
forms, and thus it is bound to the prime number 2.

This proof sparks the question of under which constructions between pro-$p$ groups the $p$-Kijima--Nishi property is preserved, and also how to compute the resulting radicals. Proposition~\ref{prop:KNFreeProduct} shows that this is the case for free pro-$p$ products. In Section~\ref{sec:Radicals of graph fundamental groups} we substantially extend this observation. The following theorem is a special case of Proposition~\ref{prop:amalg} and Proposition~\ref{prop:HNN}:
\begin{mainthm}\label{mainthm:amalg HNN}
Let $G_0$, $G_1$ and $G_2$ be pro-$p$ groups such that $\calR_p(G_i)=0$ for $i=0,1,2$. If $G$ is one of the following groups, then $\calR_p(G)=0$ and hence $G$ satisfies the $p$-Kijima--Nishi property:
\begin{enumerate}[label=(\roman*)]
     \item $G$ is a proper amalgamated free pro-$p$ product $G=G_1\amalg_H G_2$ over a common subgroup $H$ of $G_1$ and $G_2$ such that the restriction maps
     $$\res_{G_i,H}^1\colon \rmH^1(G_i,\F_p)\to\rmH^1(H,\F_p)$$ are surjective for both $i=1,2$;
     \item $G$ is a {\rm proper} pro-$p$ HNN-extension $G=\mathrm{HNN}(G_0,H,f)$ for a subgroup $H$ of $G_0$ and a monomorphism $f\colon H\to G_0$ such that the two maps
     $$f^\ast\circ\res_{G_0,f(H)}^1,\res_{G_0,H}^1\colon \rmH^1(G_0,\F_p)\longrightarrow\rmH^1(H,\F_p),$$
     are equal --- here $f^\ast\colon\rmH^1(f(H),\F_p)\to\rmH^1(H,\F_p)$ is induced by $f$ ---, and $\res_{G_0,H}^1$ is surjective;
 \end{enumerate}
\end{mainthm}
Also, we show that the assumptions on the restriction maps may not be dropped. As in the case of abstract groups, a pro-$p$ group assembled via the constructions considered in Theorem~\ref{mainthm:amalg HNN} falls under a more general construction using a profinite version of Bass--Serre theory (cf., e.g., \cite{RZ:newhorizons,Rib17}): they can be described as the fundamental pro-$p$ group of a certain graph of pro-$p$ groups.
In Section~\ref{ssec:Fundamental Groups} we generalize Theorem~\ref{mainthm:amalg HNN} to fundamental groups of certain finite graphs of pro-$p$ groups.

Finally, we consider generalized right-angled Artin pro-$p$ group and $\Delta$-right-angled Artin pro-$2$ groups. These pro-$p$ groups are a generalization of the pro-$p$ completion of {\sl right-angled Artin groups} associated to finite combinatorial graphs, introduced respectively in \cite{quadrelliQuadraticCoho} and in \cite{deltaRAAGs} (for the definition see \S-\ref{ssec:pRAAGs}--\ref{ssec:deltaRAAGs}).
We underline that very few of them occur as maximal pro-$p$ quotients of absolute Galois groups of fields containing a $\pth$-root of unity.

We determine the radicals of these groups explicitly, allowing us to conclude the following theorem.
\begin{mainthm}[Cor.~\ref{cor:RAAGs have KN} and Prop.~\ref{prop:D-RAAGs have KN}]\label{thm:RAAGs and action}
Let $p$ be a prime number.
Suppose that a pro-$p$ group $G$ is one of the following:
\begin{enumerate}[label=(\roman*)]
%    \item[(a)] a semidirect product $F(X)\rtimes G_0$, where $G_0$ is a $p$-elementary abelian group {\rm(}satisfying $|G_0|\geq p^2$ if $p\neq2${\rm )} acting on a set $X$, and $F(X)$ is the free pro-$p$ group generated by $X$ with induced $G_0$-action;
    \item $G$ is a generalized right-angled Artin pro-$p$ group;
    \item $p=2$ and $G$ is a $\Delta$-right-angled Artin pro-2 group.
\end{enumerate}
Then $G$ satisfies the $p$-Kijima--Nishi property.
\end{mainthm}
\subsection*{Structure of the paper}
\begin{itemize}
    \item In Section~\ref{sec:p Kijima Nishi and Kaplansky Radical} we first recall some elementary facts about Galois cohomology before showing that our definition generalizes the classical Kaplansky radical and the H-conjecture. Later, in Section~\ref{ssec:Examples for KN} we give several group theoretic examples illustrating the $p$-Kijima--Nishi property.
    \item Section~\ref{sec:Arithmetic and geometric fields} is devoted to a more arithmetic study of the $p$-Kijima--Nishi property. Readers more interested in group theoretic applications, might proceed to the subsequent section, as the methods used here differ substantially from the rest of the paper. Here we exhibit large classes of fields possessing the $p$-Kijima--Nishi property proving most of Theorem~\ref{mainthm:List of fields}. In Section~\ref{ssec:fields without KN} we construct fields not satisfying the $p$-Kijima--Nishi property. 
    \item In Section~\ref{sec:Elementary Type groups have KN}, we recall the definition of pro-$p$ groups of elementary type, and establish the $p$-Kijima--Nishi property for them, by showing  the stability of this property under the involved constructions of pro-$p$ groups.
    \item We continue the study of universal constructions of pro-$p$ groups in Section~\ref{sec:Radicals of graph fundamental groups} by first computing the $\FF_p$-cup radical of certain amalgamated products and HNN-extensions. We then apply this to certain fundamental groups of graphs of pro-$p$ groups.
    \item Section~\ref{sec:RAAGs and Delta-RAAGs} is devoted to the study of the $p$-Kijima--Nishi property for generalized pro-$p$ RAAGs and $\Delta$-RAAGs, which are defined purely in terms of combinatorial data. We show that all such groups satisfy the property, although not all of them arise as maximal pro-$p$ Galois groups of fields.
    \item Finally, in Section~\ref{sec:No roots of unity} we return to more arithmetic situations and discuss a potential generalization for fields which do not contain a primitive $\pth$-root of unity.
\end{itemize}

\section{The \texorpdfstring{$p$}{}-Kijima--Nishi property and the Kaplansky radical}
\label{sec:p Kijima Nishi and Kaplansky Radical}

%%%%%%%%%%%%%%%%%%%%%%%%%%%%%%%%%%%5
\subsection{The \texorpdfstring{$\FF_p$}{}-cup radical of a profinite group}

From now on, every subgroup of a profinite group is tacitly assumed to be closed with respect to the profinite topology. Therefore, sets of generators of profinite groups groups, and presentations are to be intended in the topological sense.

%Given a pro-$p$ group $G$, we write \[ [h, g] = h^{-1}\cdot h^g = h^{-1}\cdot g^{-1}\cdot h\cdot g, \qquad\text{for all } g, h \in G.\]

For basic facts on Galois cohomology of profinite and pro-$p$ groups we refer to \cite[Ch.~I, \S~1--4]{serre:galc} and to \cite[Ch.~I, \S~1--6 and Ch.~III, \S~9]{NSW2008}.
{ In particular, we recall that if $G$ is a pro-$p$ group, then the first $\F_p$-cohomology group $\rmH^1(G,\F_p)=\mathrm{Hom}_{\rm cts}(G,\Z/p)$ (cf. \cite[Ch.~I, \S~4.2]{serre:galc}) --- hence, if $\set{x_i}{,i\in I}$ is a minimal generating set of $G$, then $\rmH^1(G,\F_p)$ has a dual basis $\set{x_i^\ast}{ i\in I}$.}

Let $G$ be a profinite group. We recall that the cup product
\[
\rmH^s(G,\FF_p)\times \rmH^t(G,\FF_p)\overset{\smallsmile}{\longrightarrow} \rmH^{s+t}(G,\FF_p)
\]
is a skew-commutative bilinear map, i.e., $\beta\smallsmile\alpha=(-1)^{st}\alpha\smallsmile\beta$ for $\alpha\in\rmH^s(G,\FF_p)$ and $\beta\in\rmH^t(G,\FF_p)$ --- see, e.g., \cite[Ch.~I, \S~4]{NSW2008}.
{ Set $\bfH(G,\F_p)\coloneq \bigoplus_{n}\rmH^n(G,\F_p)$, considered as a graded, graded-commutative, $\F_p$-algebra.}

\begin{defn}
    \label{defn:Fp cup radical}
    For a profinite group $G$, we call
    \begin{align*}
        \calR_p(G)\coloneq \set{\alpha\in \rmH^1(G,\FF_p)}{\alpha \smallsmile\beta=0\text{ for all }\beta\in \rmH^1(G,\FF_p)}
        \end{align*}
    the \emph{$\FF_p$-cup radical} of the profinite group $G$.
    %If $G=G_{\KK}$ for a field $\KK$ containing a primitive $p^{\rm th}$-root of unity, then we also write $\calR_p(G_{\KK})$ instead of $\calR_p(G_\KK)$.
\end{defn}

\begin{rem}

\label{rem: radical reduction to pro-p case}
    If $G$ is a profinite group and $G(p)$ its maximal pro-$p$ quotient, then $\rmH^1(G(p),\FF_p)\cong \rmH^1(G,\FF_p)$ via the inflation morphism, and the 5-term sequence for a group extension yields that the inflation $\rmH^2(G(p),\FF_p)\to \rmH^2(G,\FF_p)$ is injective (cf., e.g., \cite[Prop.~1.6.7]{NSW2008}). 

\noindent    Since inflation commutes with the cup product, we have $\calR_p(G(p))\cong \calR_p(G)$. Thus, it suffices to consider the pro-$p$ case. 
\end{rem}

For pro-$p$ groups, one may recover the behavior of the cup product from a minimal presentation.
Recall that a presentation 
$$G=F/R=\pres{x_i,\:i\in I}{r_j=1,\:j\in J}$$
of a pro-$p$ group $G$, with $F$ the free pro-$p$ group generated by $\set{x_i}{ i\in I}$ and $R$ the normal subgroup of $F$ generated by $\set{r_j}{ j\in J}$, is minimal if the inflation gives an isomorphism $\rmH^1(F,\F_p)\simeq\rmH^1(G,\F_p)$, or, equivalently, if every $r_j$ lies in the Frattini subgroup $\Phi(F)=F^p[F,F]$ (cf. \cite{NSW2008}). Now let $F_3\subseteq F$ be the subgroup
\[F_3=\begin{cases} F^4[F,F]^2[[F,F],F] & \text{if }p=2,\\ 
F^p[[F,F],F] & \text{if } p>2
\end{cases}\]
--- i.e., $F_3$ is the third term of the Zassenhaus filtration of $F$. Then, when $G$ is finitely generated, the $\F_p$-vector space $\Phi(F)/F_3$ has a basis $\mathcal{B}$ consisting of the cosets $[x_i,x_{i'}]F_3$, $i<i'$, and also $x_i^2F_3$ if $p=2$. 
The following result is a consequence of \cite[Prop. 1.3.2]{Vogel2004}.
\begin{prop}\label{prop:cup product and shape of relations}
Let $G$ be a finitely presented pro-$p$ group, with $G=F/R$ a minimal presentation as above.
Then one has $x_i^\ast\smallsmile x_{i'}^\ast\neq 0$, with $i<i'$, if and only if $[x_i,x_{i'}]F_3$ occurs in $r_jF_3$ expressed as a linear combination of the elements of $\mathcal{B}$, for some $j$.
\end{prop}

As for the behavior of the corestriction map on the $\F_p$-cup radical of a profinite group, one has the following --- for definition and basic properties of the corestriction map see, e.g., \cite[Ch.~I, \S~4]{NSW2008}.

\begin{lem}\label{lem:maps between radicals}
\begin{enumerate}[label=(\arabic*)]
    \item \label{it:lem maps between radicals cores} Let $G$ be a profinite group and $U\leq G$ an open subgroup, then the corestriction map $\cores_{U,G}^1:\rmH^1(U,\FF_p)\to \rmH^1(G,\FF_p)$ induces a map 
    $$c_{U,G}:\mathcal{\calR}_p(U)\longrightarrow \mathcal{\calR}_p(G).$$
    \item \label{it:lem maps between radicals res} If $\varphi:G\to H$ is a homomorphism such that $\varphi^*:\rmH^1(H,\FF_p)\to \rmH^1(G,\FF_p)$ is surjective, then we have $\varphi^*(\calR_p(H))\subseteq \calR_p(G)$.
\end{enumerate}

\end{lem}
\begin{proof}
To show \ref{it:lem maps between radicals cores} let $\alpha$ be an element of $\calR_p(U)$. For $\beta\in \rmH^1(G,\FF_p)$ arbitrary, by \cite[Prop.~1.5.3]{NSW2008} one has the identity
\begin{align*}
    \cores_{U,G}^1(\alpha)\smallsmile\beta=\cores^{2}_{U,G}(\alpha\smallsmile \res_{G,U}^1(\beta))=\cores^2_{U,G}(0)=0.
\end{align*}
Therefore $\cores^1_{U,G}(\alpha)\in \calR_p(G)$.

For \ref{it:lem maps between radicals res} let $\alpha\in \calR_p(H)$ and $\beta\in \rmH^1(G,\FF_p)$. From the surjectivity of $\varphi^*$, it follows that $\beta=\varphi^*(\beta')$ for a suitable $\beta'\in \rmH^1(H,\FF_p)$. Then 
\begin{align*}
    \varphi^*(\alpha)\smallsmile \beta=\varphi^*(\alpha)\smallsmile \varphi^*(\beta')=\varphi^*(\alpha\smallsmile \beta')=0.
\end{align*}
This shows that $\varphi^*(\alpha)\in \calR_p(G)$.
\end{proof}
\begin{defn}
    \label{defin:pKN}
    A profinite group $G$ is said to have the \emph{$p$-Kijima--Nishi property} if for each open normal subgroup $U$ of $G$ of index $p$ the map $c_{U,G}:\calR_p(U)\to\calR_p(G)$ from Lemma~\ref{lem:maps between radicals} \ref{it:lem maps between radicals cores} is surjective.
\end{defn}
The following useful observation is immediate from the definition of the $p$-Kijima--Nishi property.
\begin{cor}\label{cor:trivial R}
If the $\FF_p$-cup radical $\calR_p(G)$ is trivial, then $G$ satisfies the $p$-Kijima--Nishi property.
\end{cor}

%%%%%%%%%%%%%%%%%%%%%%%%%%%%%%%%%%%%%%%%%%

\subsection{Norm, corestriction, and the Kummer isomorphism}
\label{ssec:Norm Corestriction and Kummer}
Let $\KK$ be a field and fix a separable closure $\KK^{\rm sep}$. We define $G_\KK\coloneq \Gal(\KK^{\rm sep}\!/\KK)$ to be its absolute Galois group. Now let $p$ be a prime different from $\mathrm{char}(\KK)$ and set $\mu_p$ to be the set of $\pth$-roots of unity of $\KK^{\rm sep}$. Then the Kummer map $\kappa:\KK^\times \to \rmH^1(G_\KK,\mu_p)$ is surjective with kernel $\KK^{\times p}$.

Now, if $\KK$ contains a primitive $\pth$-root of unity, then $\mu_p\cong \FF_p$ as $G_\KK$-modules and $\kappa$ induces an isomorphism $\KK^\times/\KK^{\times p}\simeq \rmH^1(G_{\KK},\FF_p)\simeq \Hom_{\rm cts}(G_\KK,\FF_p)$. If we view $\alpha\in \rmH^1(G,\FF_p)$ as a continuous homomorphism $G_\KK\to \FF_p$, then we have
\begin{align*}
    \ker(\alpha)=G_{\KK(\sqrt[p]{a})}\unlhd G_{\KK}
\end{align*}
where $a\in \KK^\times$ is any element with $\kappa(a)=\alpha$. Now, let $\LL/\KK$ be a finite separable extension. Then $G_{\LL}$ is an open subgroup of $G_\KK$ and the following square commutes:
\begin{equation}
\label{diag:cores square}
    \begin{tikzcd}
        \LL^\times \arrow[r,"\kappa"] \arrow[d,"{N_{\LL/\KK}}"]&\rmH^1(G_{\LL},\FF_p)\arrow[d,"{\cores^1_{G_\LL,G_\KK}}"]\\
        \KK^\times \arrow[r,"\kappa"] &\rmH^1(G_{\KK},\FF_p)
    \end{tikzcd}
\end{equation}

We recall the following well known result relating the cup product and the norm of radical extensions (cf., e.g., \cite[Cor.~4.7.7]{GilleSzamuely})

\begin{lem}\label{lem:norm and cup}
Let $\KK$ be a field containing a primitive $\pth$-root of unity, set $G=G_{\KK}$ and let $a,b$ be elements of $\KK^\times$.
The following are equivalent.
\begin{enumerate}[label=(\roman*)]
    \item One has $b=N_{\KK(\sqrt[p]{a})/\KK}(u)$ for some $u\in\KK(\sqrt[p]{a})^\times$;
    \item The cup product $\kappa(a)\smallsmile\kappa(b)\in\rmH^2(G,\FF_p)$ is trivial.
\end{enumerate}
\end{lem}

\begin{rem}
    \label{rem:Brauer group and cup product}
{ It is well known that fixing an isomorphism of $G_{\KK}$-modules $\mu_p\cong\F_p$ --- and thus a generator $\zeta$ of $\mu_p$ --- yields an isomorphism $\rmH^2(G_{\KK},\F_p)\cong{}_p\mathrm{Br}(\KK)$, where the former is the $p$-torsion part of the Brauer group of $\KK$, sending $\kappa(a)\smallsmile\kappa(b)$ to the class of the cyclic algebra $(a,b)_\zeta$ (see, e.g., \cite[Prop.~4.7.1]{GilleSzamuely}).} 
\end{rem}
%\todo{Add a remark on cup product pairing and Brauer groups}

\begin{rem}    \label{rem:RpG intersection cores}
By Lemma~\ref{lem:norm and cup} we infer that, for $\alpha\in \rmH^1(G_\KK,\FF_p)$,
    \begin{align*}
        \ker(c_\alpha)= \cores_{\ker(\alpha),G_{\KK}}^1(\rmH^1(\ker (\alpha),\FF_p)),
    \end{align*}
    where 
$c_\alpha\colon\rmH^1(G,\FF_p)\to\rmH^2(G,\FF_p)$ denotes the linear map $\beta\mapsto\alpha\smallsmile\beta$. Thus, we arrive at the following characterization of $\calR_p(G_\KK)$:
\begin{equation*}
    \calR_p(G) =
\bigcap_{\alpha\in\rmH^1(G,\FF_p)}\cores_{\ker(\alpha),G}^1(\rmH^1(\ker(\alpha),\FF_p))=
\bigcap_{U\lhd_p G}\cores_{U,G}^1(\rmH^1(U,\FF_p)),
\end{equation*}
where in the latter intersection $U$ runs through all open normal subgroups of $G$ of index $p$.
\end{rem}

%%%%%%%%%%%%%%%%%%%%%%%%%%%%%

\subsection{The Kaplansky radical and the \texorpdfstring{$\FF_2$}{}-cup radical}
\label{ssec:Traditional radical and F2 cup radical}

Let $\KK$ be a field { containing a primitive $p^{\mathrm{th}}$ root of unity}.
Recall that for $p=2$ the Kaplansky radical ${\calR}_{\mathrm{K}}(\KK)$ of $\KK$ is defined as
\begin{align*}
    {\calR}_{\mathrm{K}}(\KK)=\set{a\in \KK^\times}{ D_\KK\langle 1,-a\rangle =\KK^\times}=\set{a\in \KK^\times}{ \mathrm{im}(N_{\KK(\sqrt{a})/\KK}) =\KK^\times},
\end{align*}
where $D_\KK\langle 1,-a\rangle$ denotes the subgroup of $\KK^\times$ whose elements are represented by the quadratic form $X^2-aY^2$ (cf., e.g., \cite{Lam2005,Kaplansky}). 
{ 
For $p$ arbitrary, the $p$-radical $\calR_p(\KK)$ of $\KK$ is defined as
\[
\calR_p(\KK)=\bigcap_{a\in\KK^\times} N_{\KK(\sqrt[p]{a})/\KK}\left(\KK(\sqrt[p]{a})^\times\right),
\]
so that $\calR_2(\KK)=\calR_{\mathrm{K}}(\KK)$ (cf. \cite[\S~1]{DarioEngler}).}
Note that {${\KK^{\times}}^p\subseteq {\calR}_{p}(\KK)$}. The following proposition is an immediate consequence of Lemma~\ref{lem:norm and cup}.

\begin{lem}
\label{lemma:RelationBetweenRadicals}
{ Let $\kappa:\KK^\times \to \rmH^1(G_\KK,\FF_p)$ be the Kummer map. Then, $\kappa$ induces an isomorphism $\calR_{p}(\KK)/{\KK^{\times}}^p\overset{\sim}{\to}\calR_p(G_{\KK})$ --- in particular, $\calR_{\mathrm{K}}(\KK)/{\KK^{\times}}^2\overset{\sim}{\to}\calR_2(G_{\KK})$ for $p=2$.}
\end{lem}
We are now ready to translate the statement of the \lq\lq H-conjecture\rq\rq\, into cohomological terms.
\begin{prop}\label{prop:equivalence}
   Let $\KK$ be a field of characteristic not $2$.
Then the ``H-conjecture'' has positive answer for $\KK$ if and only if for any quadratic extension $\LL/\KK$ the map 
    \begin{equation*}
        c_{\LL/\KK}:\calR_2(\LL)\longrightarrow \calR_2(G_{\KK}),\qquad \alpha\longmapsto \cores_{G_{\LL},G_\KK}^1(\alpha)
    \end{equation*}
    is surjective.
\end{prop}

\begin{proof}
    Assume that the ``H-conjecture'' has positive answer  for $\KK$, and let $\LL=\KK(\sqrt{a})$ be an arbitrary quadratic extension of $\KK$. 
    Pick $\beta\in \calR_2(G_\KK)$ and $b\in \KK^\times$ such that $\kappa(b)=\beta$. Then $b\in {\calR_{\mathrm{K}}}(\KK)$ by Lemma~\ref{lemma:RelationBetweenRadicals}.
    By assumption, there exists $u\in \calR_{\rm K}(\K(\sqrt{a}))$ such that $N_{\K(\sqrt{a})/\KK}(u)=b$. Using the commutative diagram \eqref{diag:cores square} (with $p=2$), we have 
    \begin{align*}
        \beta=\cores^1_{G_{\K(\sqrt{a})},G_\KK}(\kappa(u))=c_{\LL/\KK}(\kappa(u)),
    \end{align*}
    and $\kappa(u)\in\calR_2(G_{\K(\sqrt{a})})$, again by Lemma~\ref{lemma:RelationBetweenRadicals}.
    This shows that for $\K(\sqrt{a})/\KK$ the map $c_{\K(\sqrt{a})/\K}$ is surjective as claimed. 

Conversely, assume that the map $c_{\LL/\KK}$ is surjective for every quadratic extension $\LL/\KK$. Given $a\in {\calR_{\mathrm{K}}}(\KK)$, $\kappa(a)\in\calR_2(G_{\KK})$ by Lemma~\ref{lemma:RelationBetweenRadicals}, and hence there is $\gamma=\kappa(u)\in\calR_2(G_{\LL})$, $u\in\LL^\times$, such that $$\kappa(a)=\cores_{G_\LL,G_\KK}^1(\gamma)=\cores_{G_\LL,G_\KK}^1(\kappa(u)).$$
Moreover, $u\in\calR_{\rm K}(\LL)$ by Lemma~\ref{lemma:RelationBetweenRadicals}.
The commutativity of the diagram \eqref{diag:cores square} implies that $\kappa(N_{\LL/\KK}(u))=\kappa(a)$. 
Since $\ker(\kappa)=\KK^{\times 2}$, one has $$a=N_{\LL/\KK}(u)\cdot b^2=N_{\LL/\KK}(ub)\qquad\text{for some }b\in \KK^\times,$$ as $b^2=N_{\LL/\KK}(b)$. This completes the proof.
\end{proof}

%%%%%%%%%%%%%%%%%%%%%%%%%%%%%%%%%%%%%%%%%

\subsection{Passing to the maximal pro-\texorpdfstring{$p$}{} quotient}
The following proposition shows that the $p$-Kijima--Nishi property --- similar to the $\FF_p$-cup radical --- only depends on the maximal pro-$p$ quotient of the group under consideration.
\begin{prop}
\label{prop:passing to the pro-p quotient}
    Let $p$ be a prime, $G$ a profinite group and $G(p)$ its maximal pro-$p$ quotient. Then $G$ has the $p$-Kijima--Nishi property if and only if $G(p)$ has the $p$-Kijima--Nishi property.
\end{prop}
\begin{proof}
    Let $\pi:G\to G(p)$ be the canonical projection and $U\unlhd_oG$ be of index $p$. Then $\pi(U)\unlhd_o G(p)$ is also of index $p$ by the universal property of the maximal pro-$p$ quotient. Thus, it remains to show that $\pi(U)\cong U(p)$. The kernel of $\pi$ is ${O}^p(G)$, which is equal to
    \begin{align*}
        \pres{ Q\in \mathrm{Syl}_{p'}(G)}{p'\text{ is a prime different from } p}_{\rm cl}.
    \end{align*}
    We have ${O}^p(U)\subseteq O^p(G)$, since every $p'$-Sylow subgroup of $U$ is also contained in a $p'$-Sylow subgroup of $G$. Since $G/U\cong G(p)/\pi(U)$, we have $O^p(G)=\ker \pi\subseteq U$. This shows that every $p'$-Sylow subgroup of $G$ is also contained in $U$ and therefore $O^p(U)=O^p(G)$, whence $\pi(U)\cong U(p)$. Thus, Remark~\ref{rem: radical reduction to pro-p case} implies $\calR_p(U)\cong \calR_p(\pi(U))$. The commutativity of the following diagram, which follows from \cite[Prop. 1.5.5 (ii)]{NSW2008} finishes the proof.
    \begin{equation*}
        \begin{tikzcd}[column sep=huge]
            \rmH^1(\pi(U),\FF_p)\arrow[r,"\inf^1_{\pi(U),U}","\sim"']\arrow[d,swap,"\cores^1_{\pi(U),G(p)}"]&\rmH^1(U,\FF_p)\arrow[d,"\cores^1_{U,G}"]\\
            \rmH^1(G(p),\FF_p)\arrow[r,"\inf^1_{G(p),G}","\sim"']&\rmH^1(G,\FF_p)
        \end{tikzcd}
    \end{equation*}
\end{proof}

\begin{rem}
{ By the Merkurjev-Suslin Theorem --- the ``degree-2 step'' of the celebrated Norm Residue Isomorphism Theorem ---, $\rmH^2(G_{\KK}(p),\F_p)$ is generated by cup-products of elements of degree 1 (cf., e.g., \cite[Thm.~4.5.7--Thm.~4.6.6]{GilleSzamuely}). 
Since $G_{\KK}(p)$ is a free pro-$p$ group if and only if $\rmH^2(G_{\KK}(p),\F_p)$ is trivial (cf., e.g., \cite[Prop. 3.5.17]{NSW2008}), one deduces that $G_{\KK}(p)$ is a free pro-$p$ group if and only if $\calR_p(G_{\KK}(p))=\rmH^1(G_{\KK}(p),\F_p)$ --- as remarked by Dario--Engler using the $p$-radical, see \cite[\S 2]{DarioEngler}.}
\end{rem}

\begin{rem}
    \label{rem: KN property with all subgroups}
    The proof of Proposition~\ref{prop:passing to the pro-p quotient} makes frequent use of the fact that we require the surjectivity of the corestriction only for normal subgroups $U$ of index $p$. In the pro-$p$ case, this poses no restriction, as all subgroups of index $p$ are automatically normal.

\noindent    One could also pose a stronger version of the $p$-Kijima--Nishi property, asking for the map induced by corestriction $\calR_p(U)\to \calR_p(G)$ being surjective for all closed subgroups of index $p$. In this case the Proposition~\ref{prop:passing to the pro-p quotient} does not remain valid, as the following concrete example shows:

\noindent    Let $p$ be an odd prime and set $G\coloneq \ZZ_p\rtimes\ZZ_p^\times$ with the canonical action. Then it is easy to see that $\rmH^1(G,\FF_p)\cong \FF_p$ by the five-term sequence using that $\ZZ_p^\times$ acts non-trivially on $\rmH^1(\ZZ_p,\FF_p)$. Now, since the cup product is graded-commutative we have $\calR_p(G)\cong \FF_p$. Consider the subgroup $U\coloneq (p\ZZ_p)\rtimes \ZZ_p^\times$. Then also $\calR_p(U)\cong \FF_p$ and it is not difficult to verify by explicit computation that the corestriction map $\rmH^1(U,\FF_p)\to \rmH^1(G,\FF_p)$ is trivial.

\noindent    In contrast, $G(p)\cong \ZZ_p$, which has the $p$-Kijima--Nishi property --- either by direct computation (cf. Example~\ref{ex:cyclic groups}) or using Example~\ref{ex:KNFree}.
\end{rem}

%%%%%%%%%%%%%%%%%%%%%%%%%%%%%%%%%%%%%%%%%%5
\subsection{Examples of pro-\texorpdfstring{$p$}{} groups with the \texorpdfstring{$p$}{}-Kijima--Nishi property}
\label{ssec:Examples for KN}
Here we exhibit some concrete and basic examples of pro-$p$ groups which satisfy --- or do not satisfy --- the $p$-Kijima--Nishi property.
\begin{ex}[pro-cyclic groups]
\label{ex:cyclic groups}
    Let $G$ be a non-trivial pro-cyclic pro-$p$ group, then one easily sees that $\calR_p(G)=\rmH^1(G,\FF_p)\cong \FF_p$ except if $p=2$ and $G\cong C_2$. In this case we find $\calR_p(C_2)=0$.

    The unique subgroup $U$ of $G$ of index $p$ is $G^p$. From the standard formula for corestriction one sees that if $U\neq 1$, then $\cores_{G^p,G}^1:\rmH^1(G^p,\FF_p)\to \rmH^1(G,\FF_p)$ is surjective. Thus $G$ has the $p$-Kijima--Nishi property if and only if $G$ is not isomorphic to $C_p$ for $p$ odd or $C_4$ if $p=2$. 
\end{ex}

\begin{ex}[Powerful groups]
\label{ex:powerful groups}
    Let $G$ be a finitely generated {\sl powerful} pro-$p$ group, i.e., the closed subgroup $G'\leq G$ generated by the commutators is contained in the closed subgroup $G^p\leq G$ generated by the $p$-powers, and also $G'\leq G^4$ if $p=2$ (cf., e.g., \cite[Def.~3.1]{DdSMS}). If $p=2$ suppose further that $G$ is torsion-free. 

    As $G$ is not cyclic, one has $\dim_{\FF_p}\rmH^1(G,\FF_p)\geq 2$ and the cup product induces an injective map $\Lambda^2(\rmH^1(G,\FF_p))\to \rmH^2(G,\FF_p)$ (cf. \cite[Thm.~5.1.6 and Cor~5.1.7]{SymondsWeigel}). Therefore, $\calR_p(G)=0$, and $G$ satisfies the $p$-Kijima--Nishi property by Corollary~\ref{cor:trivial R}. For example, this applies to finitely generated free abelian pro-$p$ groups or elementary abelian groups.
\end{ex}
\begin{ex}[Free pro-$p$ groups]
\label{ex:KNFree}
    Let $F$ be a profinite group of $p$-cohomological dimension $1$, then $F$ has the $p$-Kijima--Nishi property. In particular, free pro-$p$ groups have the $p$-Kijima--Nishi property. 

    Indeed, since every open subgroup $U$ of $F$ has $p$-cohomological dimension $\leq 1 $, we have $\calR_p(U)=\rmH^1(U,\F_p)$, and the statement is equivalent to showing that the corestriction
    \begin{align*}
        \rmH^1(U,\FF_p)\to \rmH^1(F,\FF_p)
    \end{align*} 
    is surjective. This is a consequence of \cite[Prop. 3.3.11]{NSW2008}.
\end{ex}

Here we give an example of a pro-$p$ group that is known not to occur as maximal pro-$p$ Galois groups of fields containing a primitive $p^{\mathrm{th}}$-root of unity and that does not have the $p$-Kijima--Nishi property.
\begin{ex}[Heisenberg group]
\label{ex:Heisenberg}
For $p$ odd, let $G$ be the \emph{Heisenberg pro-$p$ group}, namely, 
    \[\begin{split}
        G&=\pres{x,y}{[x,[x,y]]=[y,[x,y]]=1}\\
        &=\pres{x,y,z}{[x,y]=z,[x,z]=[y,z]=1}.
    \end{split}\]
Since the former is a minimal presentation, Proposition~\ref{prop:cup product and shape of relations} implies $\alpha\smallsmile\beta=0$ for every $\alpha,\beta\in\rmH^1(G,\FF_p)$, and thus $\calR_p(G)=\rmH^1(G,\FF_p)$. Since $G$ contains a two-generated abelian subgroup, $\rmH^2(G,\FF_p)\neq 0$, and hence, by the Norm Residue Isomorphism Theorem, $G$ does not occur as maximal pro-$p$ Galois groups of any field containing a primitive $p^{\mathrm{th}}$-root of unity.

\noindent
Now, consider the open subgroup $U\subseteq G$ generated by $u\coloneq x^p,y,z$. {It is straightforward to verify that $U$ is normal, has index $p$, and has minimal presentation} 
\[U=\pres{u,y,z}{[u,y]=z^p, [u,z]=[y,z]=1}\]
Therefore, $\dim \rmH^1(U,\FF_p)=3$. Since $U$ is powerful, one has $\calR_p(U)=0$ (cf. Example~\ref{ex:powerful groups}), and hence $G$ does not have the $p$-Kijima--Nishi property.
\end{ex}

%%%%%%%%%%%%%%%%%%%%%%%%%%%%%%%%%%%%%%%%%%%%%%%%%

\subsection{Testing the $p$-Kijima--Nishi property on open subgroups}
Sometimes it can be hard to check the $p$-Kijima--Nishi property directly for a profinite group, as there might not be a direct description of all normal subgroups of index $p$ (for example in the case of semidirect products). The next proposition says that it is sufficient to test the property for an open subgroup $U$ of $G$ together with the surjectivity of $c_{U,G}$ for this concrete subgroup. This is especially effective if $G$ has virtual $p$-cohomological dimension $1$, as we will show in Proposition~\ref{prop:free semidirect finite}.
\begin{prop}
\label{prop:KN property based on subgroup}
    Let $G$ be a profinite group containing an open subgroup $U$ such that $c_{U,G}:\calR_p(U)\to \calR_p(G)$ is surjective and $U$ has the $p$-Kijima--Nishi property. Then also $G$ has the $p$-Kijima--Nishi property. 
\end{prop}
\begin{proof}
    Let $V$ be an arbitrary open subgroup of index $p$ in $G$. The index $[U:V\cap U]$ is either $1$ or $p$. If it is one, then $U$ is contained in $V$ and by $c_{U,G}=c_{V,G}\circ c_{U,V}$, we see that $c_{V,G}$ is surjective. 

    Now assume that $V\cap U$ has index $p$ in $U$, and consider the following commutative diagram:
    \begin{equation*}
        \begin{tikzcd}
            \calR_p(U\cap V)\arrow[rr,"c_{U\cap V,U}",two heads]\arrow[d,"c_{U\cap V,V}"]&&\calR_p(U)\arrow[d,"c_{U,G}",two heads]\\
            \calR_p(V)\arrow[rr,"c_{V,G}"]&&\calR_p(G)
        \end{tikzcd}
    \end{equation*}
    The upper horizontal arrow is surjective, since $U$ has the $p$-Kijima--Nishi property. Thus, the composition $c_{U\cap V,G}=c_{V,G}\circ c_{U\cap V,V}$ is surjective and thereby also $c_{V,G}$, implying that $G$ has the $p$-Kijima--Nishi property.
\end{proof}
The following proposition applies the previous criterion for a semidirect product of a free profinite group with an elementary abelian group. We will use it in the proof of Theorem~\ref{thm: rational function field trivial radical} to show the $2$-Kijima--Nishi property for fields like $\RR(t)$.
\begin{prop}
\label{prop:free semidirect finite}
    Let $G=F\rtimes H$ be a semidirect product of profinite groups satisfying the following properties:
    \begin{enumerate}
        \item The maximal pro-$p$ quotient $H(p)$ of $H$ is isomorphic to $C_p^k$ with $k\geq 2 $ if $p$ is odd.
        \item $F$ has $p$-cohomological dimension $\leq 1$, and there exists a discrete $H$-set $X$ such that $\rmH^1(F,\FF_p)\cong \FF_p[X]$ as $H$-modules.
    \end{enumerate}
    Then $G$ has the $p$-Kijima--Nishi property.
\end{prop}
\begin{proof}
    For simplicity, we assume that $H$ is already pro-$p$ and hence $H\cong C_p^k$. The ideas for the general case are the same, but the arguments become substantially more technical.

    By Proposition~\ref{prop:KN property based on subgroup} it is sufficient to show that $c_{F,G}:\calR_p(F)\to \calR_p(G)$ is surjective, since $F$ has the $p$-Kijima--Nishi property (see Example~\ref{ex:KNFree}). Write $\Omega\coloneq H\backslash X$ for the set of orbits of the action of $H$ on $X$ and for $\omega\in \Omega$ write $H_\omega$ for the stabilizer of $\omega$ in $H$, which is well-defined as $H$ is abelian. 
    
    Firstly, we show that $\calR_p(G)\cong \bigoplus_{\Omega_r}\FF_p$, where $\Omega_r\coloneq \set{\omega\in \Omega}{ H_\omega=1}$. We observe that for every $n\in \NN$ there is a short exact sequence
    \[
\xymatrix{ 0\ar[r] & \rmH^n(H,\FF_p)\ar^-{\inf^n_{H,G}}[r] & \rmH^n(G,\FF_p)\ar^-{\partial_n}[r] & \rmH^{n-1}(H,\rmH^1(F,\FF_p))\ar[r] & 0}\]
    coming from the Hochschild--Serre spectral sequence, which degenerates in the second page (an alternative construction can be found in \cite[Prop. 6.8.2]{GilleSzamuely}). Notice that $\partial_1:\rmH^1(G,\FF_p)\to \rmH^1(F,\FF_p)^H$ is induced by the restriction morphism. By Shapiro's Lemma we have natural isomorphisms
    \begin{align*}
        {\rm sh}^n:\rmH^{n}(H,\rmH^1(F,\FF_p))\cong \bigoplus_{\omega \in \Omega}\rmH^{n}(H,\FF_p[\omega])\overset{\bigoplus_{\omega}{\rm sh}^n_\omega}\longrightarrow \bigoplus_{\omega\in \Omega}\rmH^n(H_\omega,\FF_p).
    \end{align*}
    Before studying the radical of $G$, we state two facts about the compatibility between the maps introduced above and the cup product. The first one being that for any $\alpha\in \rmH^n(G,\FF_p)$ and $\beta\in \rmH^m(H,\FF_p)$ we have by \cite[Lem. 6.8.4]{GilleSzamuely}
    \begin{align*}
        \partial_{n+m}(\alpha\smallsmile {\inf}^m_{H,G}(\beta))=\partial_{n}(\alpha)\smallsmile \beta.
    \end{align*}
    Furthermore, we infer that for $\alpha\in \rmH^n(H,\FF_p[\omega])$ and $\beta\in \rmH^m(H,\FF_p)$
    \begin{align*}
        {\rm sh}^{n+m}_\omega(\alpha \smallsmile \beta)={\rm sh}_\omega^n(\alpha)\smallsmile \res^m_{H,H_\omega}(\beta)
    \end{align*}
    This is an easy consequence of the functoriality of the cup product and the construction of the Shapiro isomorphism.
    
    Let $\alpha\in \rmH^1(H,\FF_p)$ be such that $\inf_{H,G}^1(\alpha)\in \calR_p(G)$. Then $\alpha\in \calR_p(H)$, which is trivial by the assumption on the structure of $H$. This implies that the natural map $\calR_p(G)\to \rmH^1(F,\FF_p)$ is injective. For $\alpha\in \calR_p(G)$ we write $\partial_1(\alpha)=\sum_{\omega\in \Omega}\alpha_\omega$ with $\alpha_\omega\in \rmH^0(H,\FF_p[\omega])=\FF_p$. Then for any $\beta\in \rmH^1(H,\FF_p)$ we have
    \begin{align*}
        0={\rm sh^2}(\partial_2(\alpha\smallsmile {\inf}^1(\beta)))={\rm sh}^2(\partial_1(\alpha)\smallsmile \beta)=\sum_{\omega\in \Omega}{\rm sh}^1_{\omega}(\alpha_\omega)\smallsmile\res_{H,H_\omega}^1(\beta).
    \end{align*}
    From the injectivity of ${\rm sh}^1_\omega$, we immediately see that if $\alpha_\omega\neq 0$, then $\res^1_{H,H_\omega}=0$, which is equivalent to $H_\omega=1$ and therefore $\omega\in \Omega_r$. This gives one of the desired inclusions. The other one follows by similar arguments.

    Now we show that the corestriction $\calR_p(F)=\rmH^1(F,\FF_p)\to \calR_p(G)$ is surjective. For that consider the following commutative diagram:
    \begin{equation*}
        \begin{tikzcd}
            \rmH^1(F,\FF_p)\arrow[r,"c_{F,G}"]\arrow[d,equal]&\calR_p(G)\arrow[d,"\res_{G,F}"]\\
            \bigoplus_{\omega \in \Omega}\FF_p[\omega]\arrow[r,"N_{H}"] &\bigoplus_{\omega\in \Omega_r}\FF_p[\omega]^H
        \end{tikzcd}
    \end{equation*}
    Here we have $N_H:\FF_p[\omega]\to \FF_p[\omega]$, $x\mapsto \sum_{h\in H}hx$. Notice that this map is zero when $H_\omega\neq 1$, and, if $H_\omega=1$, then it is surjective, as $\FF_p[\omega]^H\cong \FF_p$, which is generated by $N_H(1)$. This shows that $N_H$ is surjective, while the surjectivity of $c_{F,G}$ follows from the fact that the right vertical arrow is an isomorphism.
\end{proof}

%%%%%%%%%%%%%%%%55
%%%
%%%%%%%%%%%%%%%%%%%%%%%%%%%%%%%%

\section{Arithmetic considerations}
\label{sec:Arithmetic and geometric fields}
In this section we sometimes write $\rmH^n(\KK,A)$ instead of $\rmH^n(G_\KK,A)$ for discrete $G_\KK$ modules $A$ to simplify notation, when several fields are involved.

The goal of this section is to verify the $p$-Kijima--Nishi property for fields whose maximal pro-$p$ Galois group is not necessarily finitely generated. Most of the results presented here are analogous to the case $p=2$, which was treated in \cite{BecherLeep2014}. Since their methods are specific to the theory of quadratic forms, we resort to more cohomological techniques, which nevertheless follow a similar philosophy.

%%%%%%%%%%%%%%%%%%%%%%%%%%%5

\subsection{Local and global fields}
By \cite[Ch.~II, \S~5, Thm.~4]{serre:galc} the maximal pro-$p$
quotients of absolute Galois groups of local fields containing a primitive $\pth$ root of unity are either Demu\v{s}kin pro-$p$ groups or trivial --- for a comprehensive survey on the structure of Demu\v{s}kin groups we direct the reader to \cite[Ch.~III, \S~9]{NSW2008}.
By Corollary~\ref{cor:trivial R}, every Demu\v{s}kin group satisfies the $p$-Kijima--Nishi property, as shown by the following.

\begin{prop}\label{prop:Demushkin KN}
 Let the pro-$p$ group $G$ be a Demu\v{s}kin group.
 Then $G$ satisfies the $p$-Kijima--Nishi property.
 In particular, if $\KK$ is a local field containing a primitive $p^{\mathrm th}$ root of unity, then $G_{\KK}$ satisfies the $p$-Kijima--Nishi property.
\end{prop}

\begin{proof}
 By definition of Demu\v{s}kin groups, the cup product induces a non-degenerate pairing
 \[
  \rmH^1(G,\F_p)\times\rmH^1(G,\F_p)\longrightarrow\rmH^2(G,\F_p)\simeq\FF_p
 \]
(cf., e.g., \cite[Def.~3.9.9]{NSW2008}).
Hence $\calR_p(G)=0$, so that $G$ satisfies the $p$-Kijima--Nishi property by Corollary~\ref{cor:trivial R}.
\end{proof}
Before moving to global fields, we need the following lemma, which is a simple combination of Krasner's lemma (\cite[Lem. 12.1.1]{NSW2008}) applied to the polynomials $X^p-a$ together with Kummer duality. It is well known to experts.
\begin{lem}
\label{lem:Krasner cohomology}
    Let $(\KK,v)$ be a discretely valued field containing a primitive $p^{\rm th}$ root of unity and $\KK_v$ its completion. Then, the map $\rmH^1(\KK,\FF_p)\to \rmH^1(\KK_v,\FF_p)$ induced by the choice of an inclusion  $\KK^{\rm sep}\hookrightarrow \KK^{\rm sep}_{v}$ is surjective.
\end{lem}

The next theorem shows that the $p$-radicals of global fields are trivial. By global fields we mean finite extensions of $\QQ$ or $\FF_p(t)$. Note that for infinite extensions of $\QQ$ or $\FF_p(t)$ the statement is not true anymore, as one has $\Br(\QQ^{\rm ab})=0$ by \cite[\S 3.3 Prop. 9]{serre:galc} or $\Br(\bar{\FF_p}(t))=0$ by Tsen's theorem (cf. \cite[Thm. 6.2.8]{GilleSzamuely}).

\begin{thm}
\label{thm:Gloabal field trivial radical}
    Let $\KK$ be a global field containing a primitive $p^{\rm th}$-root of unity. Then $\calR_p(G_{\KK})=0$ and, in particular, $\KK$ has the $p$-Kijima--Nishi property.
\end{thm}
\begin{proof}
    Let $\alpha=\kappa(a)\in \calR_p(G_{\KK})$. For any non-archimedean place $v$ of $\KK$ we have, by Lemma~\ref{lem:Krasner cohomology}, that the associated local class $\alpha_v\in \rmH^1(\KK_v,\FF_p)$ satisfies
    \begin{align*}
        \alpha_v\smallsmile \beta =0 \qquad \text{for all}\quad \beta\in \rmH^1(\KK_v,\FF_p).
    \end{align*}
    Therefore, for all $v$ one has $\alpha_v\in \calR_p(G_{{\KK}_v})$, which is trivial by Proposition~\ref{prop:Demushkin KN}. Thus, for all $v$ and for $a_v\in\KK_v^\times$ such that $\kappa(a_v)=\alpha_v$, one has $a_v\in\KK_v^{\times p}$. By the Grunwald--Wang Theorem (see for example \cite[Thm.~IX.1]{ArtinTate2009}) it follows that $a\in \KK^{\times p}$, thus $\alpha=0$ in $\rmH^1(\KK,\FF_p)$.
\end{proof}
\subsection{Discretely valued fields and rational function fields}
Next, we focus on rational function fields. It would be interesting to see if the approach is generalizable for curves of higher genus.
The classical Kaplansky radical was studied for function fields of arithmetic curves by G. Manzano Flores in his PhD thesis \cite{Gonzalo2022}, where he related it to certain reduction properties of the curves.

\begin{prop}
\label{prop:Radical of discretely valued field}
    Let $(\KK,v)$ be a $p$-henselian field with residue field $\mathbb{k}$. Assume that $\mathbb{k}$ contains a primitive $p^{\mathrm{th}}$-root of unity (in particular, the characteristic of $\mathbb{k}$ is different from $p$). Then $\KK$ has the $p$-Kijima--Nishi property. If $\mathbb{k}$ is not $p$-closed, that is $G_{\mathbb{k}}(p)\neq 1$, then $\calR_p(G_\KK)=0$.\
\end{prop}

\begin{proof}
    First we notice that if $G_\mathbb{k}(p)=1$, then $G_\KK(p)\cong \ZZ_p$, since there are no unramified extensions and every $p$-extension is tamely ramified. Then the statement follows from Example~\ref{ex:KNFree}. In the other case, this is an immediate consequence of \cite[Thm. 3.6]{Wadsworth}.
\end{proof}

\begin{thm}
\label{thm: rational function field trivial radical}
    Let $p$ be a prime number, $\mathbb{k}$ a field containing a primitive $p^{\mathrm{th}}$-root of unity and set $\KK\coloneq \mathbb{k}(t)$ to be the rational function field in one variable over $\mathbb{k}$. Then $\KK$ has the $p$-Kijima--Nishi property, and, if $G_{\mathbb{k}}(p)$ is infinite, then $\calR_p(G_\KK)$ is trivial.
\end{thm}
\begin{proof}
    We recall that the \emph{Faddeev sequence} (see for example \cite[Thm. 6.9.1 resp. Cor. 6.9.3]{GilleSzamuely}) yields for all $i\geq 1$
    \begin{equation*}
    0\to \rmH^i(\mathbb{k},\FF_p)\to \rmH^i(\KK,\FF_p)\overset{\sum_v\partial_v^i}\longrightarrow \bigoplus_v\rmH^{i-1}(\mathbb{k}_v,\FF_p)\to \rmH^{i-1}(\mathbb{k},\FF_p).
    \end{equation*}

    Here, $v$ ranges over all discrete valuations on $\mathbb{k}(t)$ which are trivial on $\mathbb{k}$. Thus, they are given by non-zero prime ideals of $\mathbb{k}[t]$ or by the infinite valuation corresponding to $t^{-1}$ and $\mathbb{k}_v$ is the finite extension of $\KK$ given by the unique monic generator of the ideal of $\mathbb{k}[t]$ corresponding to $v$ or $v=\infty$ and its associated residue field is $\mathbb{k}$. The maps $\partial^i_v$ are the so called \textit{residue maps}, which are described in \cite[Const. 6.8.5]{GilleSzamuely} and factor through $\rmH^i(\KK_v,\FF_p)$, where $\KK_v$ denotes the completion of $\KK$ with respect to the valuation $v$. The last map is the sum of the corestriction maps $\rmH^{i-1}(\mathbb{k}_v,\FF_p)\to \rmH^{i-1}(\mathbb{k},\FF_p)$. Note that in contrast to the situation described in \cite[\S 6.9]{GilleSzamuely} we do not have to take into account Tate twists, as all fields under consideration contain a primitive $p^{\mathrm{th}}$-root of unity.

    We now distinguish three cases:
    \begin{itemize}[leftmargin=.5cm]
        \item \textbf{$G_\mathbb{k}(p)$ is trivial:} Then $\rmH^i(\mathbb{k},\FF_p)=0$ for $i\geq 1$ and the same holds for $\rmH^i(\mathbb{k}_v,\FF_p)$. Thus using the Faddeev sequence, we see that $G_\KK$ has $p$-cohomological dimension $1$ and we get the $p$-Kijima--Nishi property by Example~\ref{ex:KNFree}. 
        \item \textbf{$G_\mathbb{k}(p)\neq 1$ is finite:} In this case $p=2$ and $G_{\mathbb{k}}(2)\cong C_2$ by the pro-$p$ version of Artin-Schreier theorem (cf. \cite{Becker:ASp}) and furthermore, the characteristic of $\mathbb{k}$ is $0$. Now by a theorem of L. van den Dries and P. Ribenboim \cite{DriesRibenboim1984} it follows that the sequence
        \begin{align*}
            1\longrightarrow G_{\bar{\mathbb{k}}(t)}\longrightarrow G_{\KK}\longrightarrow G_{\mathbb{k}}\longrightarrow 1
        \end{align*}
        is split. We need some basic facts from étale cohomology, for which we refer to \cite{Milne1980}. We identify $G_{\bar{\mathbb{k}}(t)}=\varprojlim \pi_1^{\rm \acute{e}t}(\mathbb{P}^1_{\bar{\mathbb{k}}}\smallsetminus S,\bar\eta)$
        where $S$ ranges over the finite subsets of closed points of $\mathbb{P}^1_{\bar{\mathbb{k}}}$, which are stable under the $G_{\mathbb{k}}$ action and $\bar\eta$ is a fixed algebraic closure of the generic point. Then we get that
        \begin{align*}
            \rmH^1(\bar{\mathbb{k}} (t),\FF_p)={\varinjlim}_{S} \rmH^1_{\rm \acute{e}t}(\mathbb{P}^1_{\bar{\mathbb{k}}}\smallsetminus S,\FF_p)={\varinjlim}_{S} \FF_p[S],
        \end{align*}
        from which one can deduce that $\rmH^1(\bar{\mathbb{k}}(t),\FF_p)$ is the permutation module generated by the discrete $G_{\mathbb{k}}$-set $\mathbb{P}^1(\bar{\mathbb{k}})$. The Faddeev sequence implies that $G_{\bar{\mathbb{k}}(t)}$ has cohomological dimension $\leq 1$. Hence, Proposition~\ref{prop:free semidirect finite} is applicable and we get the $p$-Kijima--Nishi property of $\KK$.
        \item \textbf{$G_\mathbb{k}(p)$ is infinite:} In this case we have $G_{\mathbb{k}_v}(p)\neq 1$ for all places $v$ of $\KK$ and thus the corresponding completion $\KK_v$ has trivial $\FF_p$-cup  radical by Proposition~\ref{prop:Radical of discretely valued field}. 
    
        Let now $\alpha\in \calR_p(G_{\KK})$. Then for any discrete valuation, we have $\alpha_v\!\smallsmile\!\beta=0$ for every $\beta \in \rmH^1(\KK_v,\FF_p)$, by Lemma~\ref{lem:Krasner cohomology}, and thus $\alpha_v\in \mathcal{R}_p(\KK_v)$, which is trivial by Proposition~\ref{prop:Radical of discretely valued field}. { Therefore, $\alpha$ is contained in the kernel of the localization map
        \begin{align*}
            \rmH^1(\KK,\FF_p)\to {\bigoplus}_{v}\rmH^1(\KK_v,\FF_p).
        \end{align*}
        Assume that $\alpha\neq 0$ and let $\LL$ be the field corresponding to $\ker(\alpha)$, then the completion $\LL_v$ coincides with $\KK_v$ for all $v$. In particular $\LL/\KK$ is unramified at every $v$. Since every étale cover of $\mathbb{P}^1_{\mathbb{k}}$ comes from an étale cover of $\mathrm{Spec}(\mathbb{k})$ we have $\LL=\mathbb{f}\KK$ for some $p$-extension $\mathbb{f}/\mathbb{k}$ and hence for every $v$ with residue field $\mathbb{k}_v=\mathbb{k}$ we have $\LL_v=\mathbb{f}\KK_v\neq \KK_v$. This yields $\alpha=0$ and the desired statement.
        } 
    \end{itemize}
\end{proof}

Now by applying the statement of Theorem~\ref{thm: rational function field trivial radical} inductively, one deduces the following corollary:
\begin{cor}
\label{cor: function field in n variables has trivial radical}
    Let $\mathbb{k}$ be a field containing a primitive $p^{\mathrm{th}}$-root of unity. If $\KK/\mathbb{k}$ is a purely transcendental extension (not necessarily of finite transcendence degree), then $\KK$ has the $p$-Kijima--Nishi property. We have $\calR_p(G_{\KK})=0$ if the transcendence degree is at least $2$ or $G_{\mathbb{k}}(p)$ is infinite.
\end{cor}
\begin{rem}
    This yields a large class of fields for which the Elementary Type Conjecture is not applicable, as their maximal pro-$p$ Galois groups are not finitely generated, but nevertheless have the $p$-Kijima--Nishi property.
\end{rem}
\subsection{Fields not possessing the \texorpdfstring{$p$}{}-Kijima--Nishi property}
\label{ssec:fields without KN}
Similar to the case $p=2$ not every field has the $p$-Kijima--Nishi property. For a construction we follow the proof by K.J. Becher and D. B. Leep \cite[Lem. 4.7 and Lem. 4.8]{BecherLeep2014} replacing smooth conics by Severi--Brauer varieties of dimension $p-1$. For basic properties of Severi--Brauer varieties we refer to \cite[\S 5]{GilleSzamuely}.
\begin{lem}
\label{lem:big function field construction}
    Let $\KK$ be a field containing a primitive $\pth$-root of unity and $\LL$ a cyclic extension of degree $p$. Then there exists an extension $\KK'/\KK$ with the following properties:
    \begin{enumerate}
        \item \label{it:K_L/K is regular} $\KK'\cap \KK^{\rm sep}=\KK$;
        \item \label{it:LK_L/L is transcendental} $\LL\KK'/\LL$ is purely transcendental;
        \item \label{it:restriction corestriction condition} The image of the restriction $\rmH^1(\KK,\FF_p)\to \rmH^1(\KK',\FF_p)$ is contained in the image of $\cores^1_{\LL\KK',\KK'}:\rmH^1(\LL\KK',\FF_p)\to \rmH^1(\KK',\FF_p)$.
    \end{enumerate}
\end{lem}
\begin{proof}
    Let $\mathcal{B}_\LL$ be a set of representatives of the isomorphism classes of Severi--Brauer varieties of dimension $p-1$ over $\KK$ having a rational point over $\LL$. For a variety $Y/\KK$, we denote by $\KK(Y)$ the function field of $Y$. Now we set
    \begin{align*}
        \KK'\coloneq  \varinjlim_{S\subseteq \mathcal{B}_\LL\,\text{ finite}}\KK\big({\prod}_{X\in S}X\big)
    \end{align*}
    Since each Severi--Brauer variety is geometrically integral, so are their finite products. By~\cite[Chapter 3, Cor. 2.14]{Liu2006} this implies (\ref{it:K_L/K is regular}). 

    For each $X\in \mathcal{B}_\LL$ we have $X_\LL\cong \mathbb{P}_{\LL}^{p-1}$ and thus $\LL\KK(X)/\LL$ is purely transcendental, which implies the same for finite products. This implies (\ref{it:LK_L/L is transcendental}).

    For (\ref{it:restriction corestriction condition}), we write $\LL=\KK(\sqrt[p]{\alpha})$ for a suitable $\alpha\in \KK$. Let $\beta\in \rmH^1(\KK,\FF_p)$ and $\beta'\in \rmH^1(\KK',\FF_p)$ be its restriction. Then, by Lemma~\ref{lem:norm and cup}, we have $\beta'\in \im(\cores^1_{\LL\KK',\KK'})$ if and only if $\beta'\smallsmile \kappa(\alpha)=0$. Consider the cup product $\beta\smallsmile \alpha\in \rmH^2(\KK,\FF_p)$, which vanishes under the restriction to $\rmH^2(\LL,\FF_p)$, and thus corresponds to a Severi--Brauer variety $X\in \mathcal{B}_\LL$. Since $\KK'$ contains $\KK(X)$, we see that $X_{\KK'}\cong \mathbb{P}^{p-1}_{\KK'}$ and therefore the restriction of $\beta\smallsmile\alpha$ to $\rmH^2(\KK',\FF_p)$ is trivial. This finishes the proof.
\end{proof}

\begin{thm}
\label{thm:Fields without p-Kijima}
    For every field $\KK$ of characteristic different from $p$, there exists an extension $\KK_\infty/\KK$ such that $\KK_\infty$ does not have the $p$-Kijima--Nishi property. 
\end{thm}
\begin{proof}
    We can assume without loss of generality that $\KK$ has a cyclic degree $p$-extension $\LL/\KK$ and that $\KK$ contains $\zeta_p$. Otherwise, we simply replace $\KK$ by $\KK(\zeta_p,t)$. Write $\LL=\KK(\sqrt[p]{\alpha})$  and $\KK_0\coloneq  \KK$. Then we define inductively $\KK_{i+1}\coloneq  \KK_i'$ with $\KK_i'$ as in Lemma~\ref{lem:big function field construction} for the cyclic extension $\LL\KK_i/\KK_i$, which has degree $p$ as $\LL\cap \KK_i=\KK$ by (\ref{it:K_L/K is regular}). Finally, we set $\KK_\infty=\varinjlim \KK_i$.

    By Lemma~\ref{lem:big function field construction}, we see that $\LL\KK_\infty/\KK_\infty$ is purely transcendental, and thus by Corollary~\ref{cor: function field in n variables has trivial radical} we conclude $\calR_p(G_{\LL\KK_\infty})=0$. It remains to show that $\calR_p(G_{{\KK}_\infty})\neq 0$. In fact, we will show that the class of $\alpha$ defines a non-trivial element in $\calR_p(G_{\KK_\infty})$. The non-triviality is clear. 

    By the second statement of Lemma~\ref{lem:big function field construction}, we see that for each $\beta_i\in \rmH^1(\KK_i,\FF_p)$ we have $\alpha\smallsmile\res_{\KK_i,\KK_{i+1}}(\beta_i)=0$. Now, since $\rmH^1(\KK_\infty,\FF_p)=\varinjlim \rmH^1(\KK_i,\FF_p)$, we conclude that $\alpha\in \calR_p(G_{{\KK}_\infty})$. This finishes the proof.
\end{proof}
\begin{rem}
    The construction by K.J. Becher and D.B. Leep --- and, consequently, our proof --- heavily relies on the construction of highly transcendental extensions. 
    It would be interesting to investigate whether there exists an algebraic extension $\KK'/\KK$ such that $\KK$ has the $p$-Kijima--Nishi property, while $\KK'$ does not.  
\end{rem}

%%%%%%%%%%%%%%%%%%%%%%%%%%%%%%%%%%%%%%%%%%%%%%%%%%%5
%%%%%%%%%%%55
%%%%%%%%%%%%%%%%%%%%%%%%%%%%%%%%%%%%%%%%%%%%%%

\section{Pro-\texorpdfstring{$p$}{} groups of elementary type}
\label{sec:Elementary Type groups have KN}
%%%%%%%%%%%%%%%%%%%%%%%%%%%%%%%%%%%%%%
\subsection{Pro-\texorpdfstring{$p$}{} groups of elementary type and orientations}
\label{ssec:Elementary Type groups}

In order to define the family of pro-$p$ groups of elementary type, it is useful to recall the notion of an {orientation of a pro-$p$ group}.
Let $1+p\ZZ_p$ denote the multiplicative group of principal units of the ring of $p$-adic integers --- i.e., $1+p\ZZ_p=\set{1+p\lambda}{\lambda\in\ZZ_p}\subseteq \ZZ_p^\times$.
An {\sl orientation of a pro-$p$ group $G$} is a morphism of pro-$p$ groups $\theta\colon G\to1+p\ZZ_p$; the pair $(G,\theta)$ is called an {\sl oriented pro-$p$ group} (in \cite{ETC2}, an oriented pro-$p$ group is called a {\sl cyclotomic pro-$p$ pair}).

Every Demu\v{s}kin group $G$ comes endowed with a canonical orientation $\theta_G\colon G\to1+p\ZZ_p$, as described by J.~Labute in \cite[\S~3]{labute:demushkin}, which is precisely the cyclotomic character (whence the name) if $G=G_{\KK}(p)$ for some field $\KK$ containing a primitive $\pth$-root of unity --- see also \cite[\S~5.3]{qw:cyc}.
For the purposes of this work we do not need a description of such an orientation. 

\begin{defn}\label{defn:ET}\rm
 The family of pro-$p$ groups of elementary type is the smallest family of oriented pro-$p$ groups $(G,\theta)$ with $G$ finitely generated, containing:
 \begin{enumerate}[label=(\alph*)]
  \item oriented pro-$p$ groups $(F,\theta)$, where $F$ is a finitely generated free pro-$p$ group, and $\theta\colon F\to1+p\Z_p$ is any orientation,
  \item oriented pro-$p$ groups $(G,\theta_G)$, where $G$ is a Demu\v{s}kin group, and $\theta_G\colon G\to1+p\Z_p$ is its canonical orientation;
 \end{enumerate}
and satisfying the following:
\begin{enumerate}[label=(\alph*)]
\setcounter{enumi}{2}
 \item \label{it:defn ET free product} if $(G_1,\theta_1)$ and $(G_2,\theta_2)$ are of elementary type, then also $(G_1\ast_p G_2,\theta)$ --- where $\theta\colon G_1\ast_p G_2\to1+p\Z_p$ is the orientation induced by $\theta_1$ and $\theta_2$ via the universal property of free pro-$p$ products --- is of elementary type,
 \item \label{it:defn ET semidirect product} if $(G_0,\theta_0)$ is of elementary type, then also $(N\rtimes G_0,\theta)$ --- where $N\simeq\Z_p$,
 $$ghg^{-1}=h^{\theta_0(g)}\qquad\text{for all }g\in G_0,\:h\in N,$$
 and $\theta\colon (N\rtimes G_0)\to1+p\Z_p$ is the orientation satisfying $\theta\vert_{G_0}=\theta$ and $\theta(h)=1$ for all $h\in N$ --- is of elementary type.
\end{enumerate}
\end{defn}
The subsequent subsections aim to prove the following theorem:
\begin{thm}
\label{thm:ETC}
    Let $(G,\theta)$ be a group of elementary type. Then $G$ has the $p$-Kijima--Nishi property.
\end{thm}
We quickly summarize the strategy of the proof. By Proposition~\ref{prop:Demushkin KN} the statement is true for Demu\v{s}kin groups and by Example~\ref{ex:KNFree} for free pro-$p$ groups. Thus, it remains to show that the $p$-Kijima--Nishi property is preserved under the operations described in \ref{it:defn ET free product} and \ref{it:defn ET semidirect product}. This will be done in Proposition~\ref{prop:KNFreeProduct} and Corollary~\ref{cor:KNsemidirect}.

\begin{rem}
    As stated in the Introduction, in \cite{Dario} R.P.~Dario sketches a proof for a positive answer to the ``H-conjecture'' under the assumption that the maximal pro-2 Galois group of $\K$ is a pro-2 group of elementary type, considering the behavior of the Kaplansky radical with respect to free pro-2 products and cyclotomic semidirect products applied to the maximal pro-2 Galois group of the fields taken into consideration.
\end{rem}
Before proceeding to the proof, we state that groups of elementary type satisfy a much stronger property than just the $p$-Kijima--Nishi property. This is due to the fact that each open subgroup of a group of elementary type is still of elementary type:
\begin{cor}
    Let $G$ be a pro-$p$ group of elementary type, then for any open subgroup $U$ of $G$, the map $c_{U,G}:\calR_p(U)\to \calR_p(G)$ is surjective.
\end{cor}
\begin{proof}
    Choose a sequence of open subgroups $G=U_0\supseteq U_1\supseteq ...\supseteq U_r=U$, such that $[U_i:U_{i+1}]=p$ for each $i$. Then each map $\calR_p(U_{i+1})\to \calR_p(U_i)$ is surjective, as each $U_i$ --- by being of elementary type --- has the $p$-Kijima--Nishi property by Theorem~\ref{thm:ETC}. Hence, their composition is also surjective, which coincides with the claimed one, by transitivity of the corestriction.
\end{proof}
%\todo{motivate this}
We tried to find more examples of groups satisfying this stronger property, but every (finitely generated) example constructed later turned out to be (virtually) of elementary type. In the following remark, we explain this observation for certain pro-$p$ subgroups of profinite completions of limit groups:
\begin{rems}
If $G$ is a limit group (e.g. the fundamental group of an orientable closed surface), i.e., a finitely generated fully residually free group (see the definition and example in \cite{SZ26}), we can show that any finitely generated pro-$p$ subgroup of $\widehat{G}$ satisfies the $p$-Kijima--Nishi property by showing that these groups are of elementary type (see \cite{ZZ19}). It is shown in \cite[Lem. 3.13]{SZ26} (jointly with Corollary \ref{cor:acylindricalaction}) that if
$$G = A \amalg_C B$$
where $C$ is a self-centralizing procyclic subgroup of a pro-$p$ limit group $A$ (see the definition in \cite{SZ26}) and $B \simeq \ZZ_p^k$ contains $C$ as a direct summand, then for any subgroup $H$ of $G$, $H$ satisfies the $p$-Kijima--Nishi property and $H$ is of elementary type (see \cite{WZ17} and \cite[Thm. 1.4]{SZ26}). 
\end{rems}

%%%%%%%%%%%%%%%%%%%%%%%%%%%%%%%%%%%%%%

%%%%%%%%%%%%%%%%%%%%%%%%%%%%%%%%%%%%%
\subsection{Free pro-\texorpdfstring{$p$}{} products}
\label{ssec:Free pro-p and free products}

The following lemma shows that $\FF_p$-cup radicals of free pro-$p$ products of pro-$p$ groups can be completely described in terms of the radicals of their free factors. For the definition of free pro-$p$ product of an infinite family of pro-$p$ groups, we refer to \cite[Ch.~IV, \S~1]{NSW2008}.

\begin{lem}\label{lem:freeprod}
    Let $G_i$ for $i\in I$ be a family of pro-$p$ groups and let $G=\coprod_{i\in I}G_i$ be its free pro-$p$ product. Then
    \begin{align*}
        \calR_p(G)\cong \bigoplus_{i\in I}\calR_p(G_i).
    \end{align*}
\end{lem}

\begin{proof}
    By \cite[Thm.~4.1.4]{NSW2008} we have an isomorphism
    \begin{align*}
        \rho^n:\rmH^n(G,\F_p)\overset{\sim}\longrightarrow\bigoplus_{i\in I}\rmH^n(G_i,\F_p),\quad \alpha\longmapsto \sum_{i\in I}\res_{G,G_i}^n(\alpha).
    \end{align*}
    Let $\alpha\in \rmH^n(G,\F_p)$ and $\beta\in \rmH^m(G,\F_p)$ and write $\rho(\alpha)=\sum_{i\in I}\alpha_i$, $\rho(\beta)=\sum_{i\in I}\beta_i$. As the restriction maps commute with the cup product, we have
    \begin{align*}
        \rho^{n+m}(\alpha \smallsmile \beta)=\sum_{i\in I}\alpha_i\smallsmile \beta_i.
    \end{align*}
    In the light of that, it is not hard to see that $\rho^1$ induces the desired isomorphism; indeed, $\alpha\in \rmH^1(G,\FF_p)$ belongs to $\calR_p(G)$ if and only if all of its components $\alpha_i$ are contained in $\calR_p(G_i)$.
\end{proof}
As a consequence, one has the the following result for free pro-$p$ products of pro-$p$ groups satisfying the $p$-Kijima--Nishi property.

\begin{prop}\label{prop:KNFreeProduct}
    Let $G_i$ for $i\in I$ be a (not necessarily finite) family of pro-$p$ groups and $G\coloneq \coprod_{i\in I}G_i$ be its free pro-$p$ product. Then $G$ has the $p$-Kijima--Nishi property if and only if all the $G_i$ have the $p$-Kijima--Nishi property.
\end{prop}

\begin{proof}
    First, assume that each $G_i$ has the $p$-Kijima--Nishi property. 

    Let $U\leq G$ be an open subgroup of index $p$. Then by a pro-$p$ version of the Kurosh Subgroup Theorem (cf., e.g., \cite[Thm.~D.3.1]{RibesZalesskii2010} or \cite[Thm.~4.2.1]{NSW2008}) for every $i\in I$ there exists a system $S_i$ of representatives $s_i$ of the double cosets $U\backslash G/G_i$ and a free pro-$p$ subgroup $F\leq G$ of finite rank such that 
    \begin{align*}
        U\cong F\ast_p\Big(\coprod_{i\in I, s_i\in S_i}(G_i^{s_i}\cap U)\Big).
    \end{align*}
    We can assume by the second Remark following Theorem~4.2.1 in \cite{NSW2008}, that $S_i$ contains $1$. By Lemma~\ref{lem:freeprod} and Example~\ref{ex:KNFree} we have 
    \begin{align*}
        \calR_p(U)\cong \rmH^1(F,\FF_p)\oplus \Big(\bigoplus_{i\in I,s_i\in S_i}\calR_p(G_i^{s_i}\cap U)\Big).
    \end{align*}
    Notice that for any $i$ and $s_i$ one has $[G_i^{s_i}:G_i^{s_i}\cap U]\in \{1,p\}$ and also $|S_i|\in \{1,p\}$. Now, let $\alpha=\sum_{i\in I}\alpha_i\in \calR_p(G)$ be arbitrary and let $I_0$ be the finite subset of $I$ consisting of the indices $i$ such that $\alpha_i\neq 0$. Fix $i\in I_0$. If $[G_i:G_i\cap U]=1$, then we choose $\beta_i=\alpha_i\in \calR_p(G_i\cap U)=\calR_p(G_i)$; otherwise, $G_i\cap U$ has index $p$ in $G_i$, and, since $G_i$ has the $p$-Kijima--Nishi property, we can choose $\beta_i\in \calR_p(G_i\cap U)$ such that 
    \begin{align*}
        \cores^1_{G_i\cap U,G_i}(\beta_i)=\alpha_i.
    \end{align*}
    Now, for $\beta\coloneq \sum_{i\in I}\beta_{i}\in \calR_p(U)$, one has $\cores^1_{U,G}(\beta)=\alpha$. This can be checked for each summand separately. For the ones with $G_i=G_i\cap U$ there is nothing to check. For the others, we have $G_iU=G$ and the statement follows from the commutative square (cf. \cite[Cor. 1.5.8]{NSW2008}):
    \begin{equation*}
        \begin{tikzcd}[column sep=huge]
    \rmH^1(U,\FF_p)\arrow[r,"\res_{U,G_i\cap U}^1"]\arrow[d,"\cores_{U,G}^1"]&\rmH^1(G_i\cap U,\FF_p)\arrow[d,"\cores_{G_i\cap U,G_i}^1"]\\
     \rmH^1(G,\FF_p)\arrow[r,"\res_{G,G_i}^1"]&\rmH^1(G_i,\FF_p)
        \end{tikzcd}
    \end{equation*}
    This shows that $c_{U,G}:\calR_p(U)\to \calR_p(G)$ is surjective. Since $U$ was arbitrary we conclude the $p$-Kijima--Nishi property.

    To show the converse implication fix $i_0\in I$, and put $A=G_{i_0}$ and $B=\coprod_{i\neq i_0}G_i$. Denote by $N$ the normal closure of $B$ in $G$. 
    If $U$ is an open subgroup of $A$ of index $p$, then the subgroup $V=U\cdot N$ is an open subgroup of $G$ of index $p$ and $V\cap A=U$, and hence $\cores_{V,G}:\calR_p(V)\to \calR_p(G)$ is surjective. Since $\res^1_{G,A}$ is surjective, we conclude that $c_{U,A}$ is surjective from the same diagram as above.
\end{proof}
In practice, it is not straightforward to check that a pro-$p$ group decomposes as a free pro-$p$ product of certain subgroups. As in the discrete case, Bass--Serre theory can be helpful to identify these situations.

The action of a pro-$p$ group $G$ on a profinite tree $T$ (see \cite[\S 2.4]{Rib17} for a definition) is said to be {\it $k$-acylindrical} if the stabilizer of any geodesic $[v,w]$ of length $>k$ is trivial for a fixed $k$.

Proposition 3.5 of \cite{WZ17} shows that $G$ acting $k$-acylindrically on a tree $T$ splits as a free pro-$p$ product of its vertex stabilizers. From this observation we infer the following corollary.
\begin{cor}\label{cor:acylindricalaction}
Let $G$ be a finitely generated pro-$p$ group. If there is a profinite tree $T$ on which $G$ acts $k$-acylindrically for some $k\geq 0$ and all the vertex stabilizers satisfy the $p$-Kijima--Nishi property, then $G$ also has the $p$-Kijima--Nishi property.
\end{cor}
Finitely generated pro-$p$ subgroups $H$ of the profinite completion $\widehat{G}$ of the fundamental group of a compact $3$-manifold $M$ are examples of finitely generated pro-$p$ groups acting $k$-acylindrically on a profinite tree with action induced by the action of $\widehat{G}$ (see \cite{WZ17}). So, verifying if $H$ satisfies the $p$-Kijima--Nishi property can be reduced to verifying if the $H$-vertex stabilizers satisfy the $p$-Kijima--Nishi property. For a general compact orientable $3$-manifold $M$, \cite[Thm. 1.3]{WZ17} provides a full description of $H$ as well as of the $H$-vertex stabilizers.

For instance, let $M$ be $S^3 \setminus K$ where $K$ is the figure-eight knot. Then $M$ is orientable hyperbolic with cusps (see \cite[\S 4.3]{Thu22}). Let $p$ be a prime and $P$ be a $p$-Sylow subgroup of $\widehat{\pi_1(M)}$. If $H$ is a finitely generated subgroup of $P$, then $H$ has the $p$-Kijima--Nishi property. Indeed, if $H$ is abelian, then it has the $p$-Kijima--Nishi property by Example~\ref{ex:powerful groups} or Example~\ref{ex:cyclic groups} since $\widehat{\pi_1(M)}$ is torsion free. 

Otherwise, by \cite[Thm. 1.2]{WZ17A} it splits as
$$H = \coprod_i A_i$$
where each $A_i$ is a free abelian pro-$p$ group of rank $\leq 2$. Hence, by Proposition~\ref{prop:KNFreeProduct} we see that $H$ has the $p$-Kijima--Nishi property also in this case.

By \cite[Cor. 1.2]{LSZ24} there is at least one Sylow subgroup that is non-abelian.

%%%%%%%%%%%%%%%%%%%%%%%%%%%%%%%%%%%%%%

\subsection{Cyclotomic semidirect products}
The following proposition shows that so-called \emph{$p$-almost direct products} have trivial $\FF_p$-cup radical. A semidirect product $N\rtimes Q$, is called $p$-almost direct, if $Q$ acts trivially on $\rmH^1(N,\FF_p)$. If $N$ is already a pro-$p$ group, this is equivalent to asking that the action of $Q$ on $N/\Phi(N)$ is trivial. Note that any direct product is $p$-almost direct.
\begin{prop}
\label{prop:almost direct product}
    Let $G=N\rtimes Q$ be a $p$-almost direct product of profinite groups such that $\rmH^1(Q,\FF_p)$ and $\rmH^1(N,\FF_p)$ are non-trivial, then $\calR_p(G)=0$.
\end{prop}
\begin{proof}
    Consider the following Hochschild--Serre spectral sequence 
    \begin{align*}
        E_2^{s,t}=\rmH^s(Q,\rmH^t(N,\FF_p))\Longrightarrow \rmH^{s+t}(G,\FF_p).
    \end{align*}
    Since $G$ is a semidirect product we get by \cite[Prop. 2.4.5]{NSW2008} that all the differentials $d^{*,1}_2$ are trivial. This implies that 
    \begin{align*}
        \rmH^1(G,\FF_p)\cong \rmH^1(Q,\FF_p)\oplus \rmH^1(N,\FF_p)^Q\cong \rmH^1(Q,\FF_p)\oplus \rmH^1(N,\FF_p).
    \end{align*}
    Note that this splitting is canonical, since $G$ is a semidirect product. Furthermore, on $\rmH^2(G,\FF_p)$ there is a decreasing filtration $F^*\rmH^2(G,\FF_p)$ such that the following sequence is exact:
    \begin{align*}
        0\to \rmH^2(Q,\FF_p)\to F^1\rmH^2(G,\FF_p)\overset{\pi}\to \rmH^1(Q,\rmH^1(N))\to 0
    \end{align*}
    The canonical map $\rmH^1(Q,\FF_p)\otimes \rmH^1(N,\FF_p) \to\rmH^1(Q,\rmH^1(N,\FF_p))$ is injective and thus the same holds for the multiplication map $E_2^{1,0}\otimes E_2^{0,1}\to E_2^{1,1}$.

    Now, assume that $\alpha\in \calR_p(G)$ and write $\alpha=\alpha_Q+\alpha_N$. For every $\beta=\beta_Q+\beta_N\in \rmH^1(G,\FF_p)$ we have $\alpha\smallsmile \beta=0$ and hence in particular $\alpha\smallsmile \beta\in F^1\rmH^2(G,\FF_p)$. We have 
    \begin{align*}
        \pi(\alpha\smallsmile \beta)=\alpha_Q\otimes \beta_N-\beta_Q\otimes \alpha_N.
    \end{align*}
    Now if $\alpha_Q\neq 0$, then we choose $\beta=0+\beta_N$ for some $\beta_N\neq 0$ and if $\alpha_N\neq 0$, we choose $\beta=\beta_Q+0$ for some $\beta_Q\neq 0$. In either of the cases $\pi(\alpha\smallsmile \beta)\neq 0$, contradicting $\alpha \smallsmile \beta=0$. Therefore $\alpha=0$, which finishes the proof.
\end{proof}
\begin{cor}
\label{cor:KNsemidirect}
    Let $G_0$ be a non-trivial pro-$p$ group (not necessarily finitely generated)  and $\theta:G_0\to 1+p\ZZ_p$ an orientation, and set $G=N\rtimes_\theta G_0$ for a free abelian pro-$p$ group $N$. Then $\calR_p(G)=0$, and therefore $G$ satisfies the $p$-Kijima--Nishi property.
\end{cor}
\begin{proof}
    Since $\theta(G_0)\subseteq 1+p\ZZ_p$ the action of $G_0$ on $\rmH^1(N,\FF_p)$ is trivial and thus the conditions of Proposition~\ref{prop:almost direct product} are satisfied.
\end{proof}

%%%%%%%%%%%%%%%%%%%%%%%%%%%%%%%%%%%%%%%%%%555
%%%%%%%%%%555

\section{The \texorpdfstring{$\FF_p$}{}-cup radical of free profinite constructions}
\label{sec:Radicals of graph fundamental groups}

In this section we will investigate the behavior of the $\FF_p$-cup radical in free profinite constructions and, in particular, address a natural question arising from Proposition \ref{prop:KNFreeProduct} regarding the behavior of these constructions with respect to  the $p$-Kijima--Nishi property. In particular, this section provides a method to construct new examples of groups satisfying the $p$-Kijima--Nishi property.

\subsection{Amalgamated products and $\HNN$-extensions}

For an overview on amalgamated free pro-$p$ products and HNN-extensions of pro-$p$ groups, we direct the reader to \cite[\S 9.2--9.4]{RibesZalesskii2010}.

First, we recall the definition of amalgamated free pro-$p$ product.
\begin{defn}
Let $G_1 = \pres{V_1}{ R_1}$ and $G_2 = \pres{ V_2 }{ R_2}$ be pro-$p$ groups and $H$ be a pro-$p$ group with injective continuous homomorphisms $f_1:H \to G_1$ and $f_2:H \to G_2$. The \emph{free product of $G_1$ and $G_2$ with amalgamated subgroup $H$}, denoted $G_1 \amalg_H G_2$, is defined by the presentation:
$$G_1 \amalg_H G_2 = \pres{V_1 \cup V_2 }{ R_1 \cup R_2,\ f_1(h)=f_2(h) \text{ for all } h \in H}.$$
The amalgamated product is called \emph{proper} if the canonical maps $G_i\to G_1\amalg_HG_2$  are injective.
\end{defn}
If $G=G_1\amalg_HG_2$ is a proper amalgamated product, then one has the following Mayer--Vietoris exact sequence (see \cite[Prop. 9.2.13 (a)]{RibesZalesskii2010}).
\begin{equation*}%\label{eq:MayerVietoris amalg}
\begin{tikzpicture}[baseline=(current bounding box.center), descr/.style={fill=white,inner sep=2pt}]
        \matrix (m) [
            matrix of math nodes,
            row sep=3.5em,
            column sep=2.6em,
            text height=1.5ex, text depth=0.25ex
        ]
        { 0\to\rmH^1(G,\F_p) &\rmH^1(G_1,\F_p)\oplus\rmH^1(G_2,\F_p) &\rmH^1(H,\F_p)\\
         \rmH^2(G,\FF_p) & \rmH^2(G_1,\FF_p)\oplus\rmH^2(G_2,\FF_p) &  \cdots \\
           };

        \path[overlay,->, font=\scriptsize,>=latex]
        (m-1-1) edge node[auto] {$r_1$} (m-1-2)
        (m-1-2) edge node[auto] {$\rho^1$} (m-1-3)
        (m-1-3) edge[out=355,in=175] node[descr,yshift=0.3ex] {$\delta_1$} (m-2-1)
        (m-2-1) edge node[auto] {$r_2$} (m-2-2) 
        (m-2-2) edge node[auto] {$\rho^{2}$} (m-2-3);
\end{tikzpicture}
\end{equation*}
Here, the maps $r^n$ and $\rho^n$ are defined as $\res^n_{G,G_1}\oplus \res_{G,G_2}^n$ and $\res^n_{G_1,H}-\res^n_{G_2,H}$, respectively.
\begin{prop}\label{prop:amalg}
	 Let $G=G_1\amalg_H G_2$ be a proper free amalgamated pro-$p$ product and  suppose that the restriction maps $\rmH^1(G_i,\FF_p)\to \rmH^1(H,\FF_p)$ for $i=1,2$ are surjective.
	 Then, the diagram
     \begin{equation*}
     \xymatrix{
	       \calR_p(G)\ar[d]_{\res^1_{G,G_1}} \ar[r]^{\res^1_{G,G_2}}&\calR_p(G_2)\ar[d]^{\res^1_{G_2,H}}\\
	       \calR_p(G_1)\ar[r]_{\res^1_{G_1,H}}&\rmH^1(H,\FF_p)
	   }
     \end{equation*}is a pull-back square.
\end{prop}
\begin{proof}
    First, notice that the assumption on the surjectivity of $\res_{G_i,H}^1$, $i=1,2$ implies that also the maps $\res_{G,G_i}^1$, $i=1,2$, and $\rho^1$ are surjective. Indeed, fix $i,j\in\{1,2\}$, $i\neq j$. For $\alpha\in \rmH^1(G_i,\FF_p)$, pick $\alpha'\in \rmH^1(G_j,\FF_p)$ such that $\res_{G_j,H}^1(\alpha')=\res^1_{G_i,H}(\alpha)$ --- such an $\alpha'$ exists as $\res_{G_j,H}^1$ is surjective. Thus, $$\alpha+\alpha'\in\ker(\rho^1)=\im(r^1),$$ and $\alpha\in\im(\res_{G,G_i}^1)$. Similarly, for $\beta\in\rmH^1(H,\FF_p)=\im(\res_{G_i,H}^1)$ pick $\alpha\in\rmH^1(G_i,\FF_p)$ such that $\alpha\vert_H=\beta$ --- then $\beta=\rho^1(\alpha,0)\in\im(\rho^1)$.
    
    From the former argument, it is easy to deduce that $\res^1_{G,G_i}(\calR_p(G))\subseteq \calR_p(G_i)$ for $i=1,2$ and thus $r^1$ restricted to $\calR_p(G)$ yields an inclusion 
    $$r^1\vert_{\calR_p(G)}\colon\calR_p(G)\hookrightarrow \calR_p(G_1)\oplus \calR_p(G_2).$$ 
    
    Thus, it remains to show that
    \begin{align*}
        r^1(\calR_p(G))\supseteq (R_p(G_1)\oplus R_p(G_2))\cap \ker(\rho^1).
    \end{align*}
    For that let $\alpha_1+\alpha_2\in \calR_p(G_1)\oplus  \calR_p(G_2)$ be in the kernel of $\rho^1$.
    By the Mayer--Vietoris exact sequence, $\alpha_1=\alpha\vert_{G_1}$ and $\alpha_2=\alpha\vert_{G_2}$ for some $\alpha\in \rmH^1(G,\FF_p)$. We show that $\alpha\in \calR_p(G)$. For $\beta\in \rmH^1(G,\FF_p)$ arbitrary, we infer
    \begin{align*}
        r^2(\alpha\smallsmile \beta)=\res_{G,G_1}^1(\alpha)\smallsmile \res_{G,G_1}^1(\beta)+\res_{G,G_2}^1(\alpha)\smallsmile \res_{G,G_2}^1(\beta)=0.
    \end{align*}
Since $\rho^1$ is surjective, $r^2$ is injective by the Mayer--Vietoris exact sequence, so that $\alpha\smallsmile \beta=0$. Since $\beta$ was arbitrary, it follows that $\alpha\in \calR_p(G)$.
\end{proof}
The following example shows that the conditions on the restriction maps posed in Proposition~\ref{prop:amalg} are necessary. It also exhibits a new example of a group, that doesn't have the $p$-Kijima--Nishi property. 
\begin{ex}
	For $p \neq 2$, let $G_1=\pres{x,y}{[x,y]=1}\simeq \ZZ_p^2\simeq G_2=\pres{u,v}{[u,v]}$ and $H$ be the subgroup of $G_2$ generated by $u$, and let $H\to G_1$ be the map $u\mapsto y^p$. Since $H$ is procyclic, the amalgamated free pro-$p$ product $G=G_1\amalg_H G_2$ is proper.
	
 	By the Mayer--Vietoris sequence, we deduce that $\dim\rmH^1(G,\FF_p)=3$, and that $\dim\rmH^2(G,\FF_p)=2$. Thus \[G=\pres{x,y,v}{[x,y]=[v,y^p]=1}\] is a minimal presentation of $G$. 
 	
 	Since $\calR_p(G_i)=0$, the pull-back of these spaces is trivial, while Proposition~\ref{prop:cup product and shape of relations} implies that $\calR_p(G)=\FF_p\cdot v^\ast$ is one-dimensional. Notice the hypotheses of Proposition~\ref{prop:amalg} are not satisfied, as $\res^1_{G_2,H}=0$.

 	Consider now $U=\ker(\phi)$, where $\phi:G\to \Z/p$ is defined by $\phi(x)=\phi(v)=0$ and $\phi(y)=1$. 
 	Then, $U$ is generated as a pro-$p$ group by $x,y^p$, and by the elements $w_h:=v^{y^h}$, $h=0,\dots,p-1$. In fact one has the following minimal presentation:
    \[U=\pres{x,y^p,w_0,\dots,w_{p-1}}{[x,y^p]=[w_h,y^p]=1,\ h=0,\dots,p-1}\]
    Verifying that it is indeed a minimal presentation is quite involved. From that one concludes that $U$ is isomorphic to the right-angled Artin pro-$p$ group on a connected graph with $p+2$ vertices, and hence $\calR_p(U)=0$ by Lemma \ref{lem:radical of RAAG}. In particular, $G$ does not satisfy the $p$-Kijima--Nishi property. 
\end{ex}
Next, we recall the definition of HNN-extensions of pro-$p$ groups.
\begin{defn}
Let $G_0 = \pres{V }{ R}$ be a pro-$p$ group, $H \subseteq G_0$ a closed subgroup and $f:H \to G_0$ an injective continuous homomorphism. The \emph{HNN-extension of $G_0$ with associated subgroup $H$ via $f$}, denoted $\HNN_f(G_0,H,t)$, is defined by the presentation:
$$\HNN_f(G_0,H,t) = \pres{ V, t}{R, tht^{-1} = f(h) \text{ for all } h \in H },$$
where the distinguished element $t$ is called the \emph{stable letter}.

If the canonical homomorphism $G_0\to \HNN_f(G_0,H,t)$ is injective, we say that the HNN-extension is \emph{proper}.
\end{defn}
If $G=\HNN_f(G_0,H,t)$ is proper, there exists a Mayer--Vietoris exact sequence
\begin{equation*}
\begin{tikzpicture}[baseline=(current bounding box.center), descr/.style={fill=white,inner sep=2pt}]
        \matrix (m) [
            matrix of math nodes,
            row sep=3.5em,
            column sep=2.6em,
            text height=1.5ex, text depth=0.25ex
        ]
        { 0\to\FF_p &\rmH^1(G,\F_p) &\rmH^1(G_0,\F_p) &\rmH^1(H,\F_p)\\
         &\rmH^2(G,\FF_p) & \rmH^2(G_0,\FF_p) &  \cdots \\
           };

        \path[overlay,->, font=\scriptsize,>=latex]
        (m-1-1) edge node[auto] {$\delta_0$} (m-1-2)
        (m-1-2) edge node[auto] {$r_1$} (m-1-3)
        (m-1-3) edge node[auto] {$\phi_f^1$} (m-1-4)
        (m-1-4) edge[out=355,in=175] node[descr,yshift=0.3ex] {$\delta_1$} (m-2-2)
        (m-2-2) edge node[auto] {$r_2$} (m-2-3) 
        (m-2-3) edge node[auto] {$\phi_f^{2}$} (m-2-4);
\end{tikzpicture}
\end{equation*}
where $r^n=\res_{G,G_0}^n$ and $\phi^n_f\coloneq \res^n_{G_0,H}-f^\ast\circ\res^n_{G_0,f(H)}$ (cf. \cite[Prop. 9.4.2 (a)]{RibesZalesskii2010}).

\begin{prop}\label{prop:HNN}
	Let $G=\HNN_f(G_0,H,t)$ be a proper pro-$p$ HNN-extension such that the restriction mapping $\res^1_{G_0,H}:\rmH^1(G_0,\FF_p)\to \rmH^1(H,\FF_p)$ is surjective, and the following triangle commutes: 
		\[\xymatrixcolsep{0cm}\xymatrix{
			\rmH^1(G_0,\FF_p)\ar[rr]^{\res^1_{G_0,H}}\ar[rd]^{\res^1_{G_0,f(H)}}&&\rmH^1(H,\FF_p)\\
			& \rmH^1(f(H),\FF_p)\ar[ru]_{f^\ast}&
		}\]
    Then $r^1$ induces an inclusion $\calR_p(G)\hookrightarrow \calR_p(G_0)\cap \ker(\res^1_{G_0,H})$. 
\end{prop}

\begin{proof}
Since $\phi^1_f=\res^1_{G_0,H}-f^*\res^1_{G_0,f(H)}=0$, the Mayer--Vietoris exact sequence yields the two exact sequences 
	\begin{align*}
        \begin{split}
            \xymatrix{0 \ar[r] & \FF_p\ar[r]^-{\delta^0} & \rmH^1(G,\FF_p)\ar[r]^-{r^1}& \rmH^1(G_0,\F_p)\ar[r] & 0},\\
            \xymatrix{ 0\ar[r]& \rmH^1(H,\FF_p)\ar[r]^-{\delta ^1}& \rmH^2(G,\F_p)\ar[r]^-{r^2} &\rmH^2(G_0,\F_p)}.
        \end{split}
	\end{align*}
    Let $\tau\in \rmH^1(G,\FF_p)$ be the dual to the stable letter $t\in G$, then $\tau$ is a generator of $\im(\delta^0)=\ker(r^1)$, since $r^1(\tau)=0$. Notice that $\tau\not\in \calR_p(G)$, since for every $h\in H$ one has a relation $[h,t]\in \Phi(G_0)$ and thus, by Proposition~\ref{prop:cup product and shape of relations}, we infer $\tau\smallsmile \alpha\neq 0$ for each $\alpha\in \rmH^1(G,\FF_p)$ such that $\res_{G,H}^1(\alpha)\neq 0$. This also shows that $\calR_p(G)\subseteq \ker(\res_{G,H}^1)$.

    From the surjectivity of $r^1$ we deduce that $r^1(\calR_p(G))\subseteq \calR_p(G_0)$ using Lemma~\ref{lem:maps between radicals} \ref{it:lem maps between radicals res}, and since $\ker(r^1)=\FF_p\tau$ intersects $\calR_p(G)$ trivially, the induced map $$\calR_p(G)\longrightarrow \calR_p(G_0)\cap \ker(\res_{G_0,H}^1)$$ is injective as claimed.
\end{proof}
\begin{rem}
    The inclusion in Proposition~\ref{prop:HNN} is in general strict, as one can see for $G_0$ the free pro-$p$ group generated by $x,y,a$, $H=\langle a\rangle\cong \ZZ_p $ and $f:H\to G_0$ defined by $a\mapsto a[x,y]$. Setting $G\coloneq \HNN_f(G_0,H,t)$, it is not hard to see that $G$ is a Demu\v{s}kin group. From this one can deduce that the HNN-extension is proper. Furthermore, it readily follows that $\calR_p(G)=0$, whereas  $\calR_p(G_0)\cap \ker(\res_{G_0,H}^1)=\FF_px^*\oplus \FF_py^*$.
\end{rem}
Similar to before, we give an example that shows that the surjectivity of the restriction map in Proposition~\ref{prop:HNN} is necessary:
\begin{ex}
    Let $G_0=\ZZ_p^2$ generated by $x$ and $y$. Let $H\cong \ZZ_p$ be the subgroup generated by $x^p$ and define $f:H\to G_0$ by $f(x^p)=y^p$. Then the restriction maps $\res^1_{G_0,H}$ and $\res^1_{G_0,f(H)}$ are both trivial. Now define $G=\HNN_f(G_0,H,t)$. From the Mayer--Vietoris sequence one deduces that 
    \begin{align*}
        G=\pres{x,y,t}{[x,y]=1,[x^p,t]=(x^{-1}y)^p}
    \end{align*}
    is a minimal presentation. From that we conclude that $\calR_p(G)=\FF_p
    \cdot t^*\neq 0$. This shows that the statement of Proposition~\ref{prop:HNN} cannot be true in this context, as $\calR_p(G_0)=0$.
\end{ex}
\subsection{Fundamental groups of graphs of pro-$p$ groups}
\label{ssec:Fundamental Groups}

Previously in this section, we analyzed the structure of the $\FF_p$-cup radical in amalgamated pro-$p$ products and pro-$p$ \HNN-extensions. Both are special cases of a more general construction, originally due to H. Bass and J-P. Serre~\cite{Serre2003} for abstract groups, and to D.~Gildenhuys, O.~Mel'nikov, L.~Ribes and P.~Zalesskii \cite[\S 6.7]{Rib17} in the profinite setting. A natural next step, therefore, is to understand the behavior of the $\FF_p$-cup radical in this construction, namely, the {\it fundamental pro-$p$ group of a finite graph of pro-$p$ groups}.

We adopt Serre's definition of a finite graph. While the definition by L. Ribes and P. Zalesskii (see \cite{SSS12}) yields an oriented graph that requires introducing formal symbols to represent the reverse direction of an edge, Serre's definition inherently includes both directions while keeping the construction independent of a chosen orientation.

\begin{defn}
{\it A finite (Serre) graph} $\Gamma$ consists of a finite set $V(\Gamma)$, a finite set $E(\Gamma)$ and two maps
$$E(\Gamma) \to V(\Gamma) \times V(\Gamma), \quad e \mapsto (d_0(e), d_1(e))\quad 
\text{and}\quad 
E(\Gamma) \to E(\Gamma), \quad e \mapsto \bar{e}$$
satisfying the following conditions: for each $e \in E(\Gamma)$ we have $\bar{\bar{e}} = e$, $\bar{e} \neq e$ and $d_0(e) = d_1(\bar{e})$. A finite graph is {\it connected} if any two vertices can be connected by a path $e_1\cdots e_n$, i.e., they are extremities of a sequence of adjacent edges; a connected graph is a {\it tree} if it has no reduced closed path $e_1\cdots e_n$, i.e., a path starting and ending at the same vertex such that $e_{i+1} \neq \bar{e}_i$ for all $i=1,\dots,n-1$. 
\end{defn}

\begin{defn}
A {\it finite graph of pro-$p$ groups} is a pair $(\calG,\Gamma)$ where $\Gamma$ is a connected finite Serre graph, $\calG$ is a family of pro-$p$ groups $\calG(m)$ for each $m \in \Gamma$ such that $\calG(e) = \calG(\bar{e})$ and, for each edge $e \in E(\Gamma)$, there are continuous monomorphisms $\partial_i = \partial_{i,e}: \calG(e) \to \calG(d_i(e))$, for $i=0,1$, and $\partial_{0,e} = \partial_{1,\bar{e}}$.
\end{defn}

\begin{defn}
Let $(\calG,\Gamma)$ be a finite graph of pro-$p$ groups and $\mathrm T$ be a maximal subtree of $\Gamma$. Then the {\it fundamental pro-$p$ group} of $(\calG,\Gamma)$ is defined by
$$\Pi_1(\calG,\Gamma) = \bigg(F\amalg \coprod_{v \in V(\Gamma)} \calG(v)\bigg)/N$$
where $F$ is the free pro-$p$ group with basis $\set{t_e}{ e \in E(\Gamma)}$ and $N$ is the closed normal closure of the set
$$\set{t_e}{e \in E(\mathrm T)} \cup \set{t_{\bar{e}}t_e}{e \in E(\Gamma)} \cup \set{\partial_{0,e}(x)^{-1}t_e\partial_{1,e}(x)t_e^{-1}}{ x \in \calG(e), e \in E(\Gamma)}.$$
Note that the image of $t_e\in \Pi_1(\calG,\Gamma)$ is trivial when $e\in E(\mathrm T)$, and $\Pi_1(\calG,\Gamma)$ is independent of the choice of $\mathrm T$, up to isomorphism.
\end{defn}

\begin{rem}
Unlike in the abstract case, a vertex group may not embed into the fundamental pro-$p$ group. When the embedding occurs, we will say that the finite graph of pro-$p$ groups is {\it proper}. Replacing the vertex groups with their images, we can construct another finite graph of pro-$p$ groups (over the same original graph) which is proper and whose fundamental pro-$p$ group is isomorphic to the original. For a more detailed explanation, we refer to \cite[\S~6.4]{Rib17}.
\end{rem}

The fundamental pro-$p$ group of a finite graph of pro-$p$ groups generalizes many free pro-$p$ constructions. For example, let $(\calG,\Gamma)$ be the graph of groups depicted as follows
$$\begin{tikzpicture}[auto]
\node[circle,draw=black, fill=black, inner sep=0pt, minimum size=1.5mm] at (0,0) {};
\node[circle,draw=black, fill=black, inner sep=0pt, minimum size=1.5mm] at (2,0) {};
\draw node[above left] {$\calG(v)$} (0,0) -- node[below] {$\calG(e)$} (2,0) node[above right] {$\calG(w)$};
\end{tikzpicture}$$
with continuous monomorphisms $\calG(e)\to \calG(v)$ and $\calG(e)\to \calG(w)$. Then one finds $\Pi_1(\calG,\Gamma)\cong \calG(v) \amalg_{\calG(e)} \calG(w)$ and the amalgamated product is proper if and only if $(\calG,\Gamma)$ is proper.

If we consider the following graph of groups $(\calG,\Gamma)$ 
\begin{center}
    \begin{tikzpicture}[scale=1.2]
    \tikzset{every loop/.style={}}
    \clip (-2.5,-.5) rectangle + (3.7,1);
    \node[circle,draw=black, fill=black, inner sep=0pt, minimum size=1.5mm] at (0,0) {};
    \path[-] (0,0) edge[scale=5,in=-150,out=-210,loop] node[left ] {$\calG(e)$} (0,0) node[right] {$\calG(v)$};
    \end{tikzpicture}
\end{center}
with continuous morphisms $\partial_i:\calG(e)\to \calG(v)$. Define $f:\partial_1(\calG(e))\to \calG(v)$ by $f(\partial_1(x))=\partial_0(x)$. Then we get 
\begin{align*}
    \Pi_1(\calG,\Gamma)\cong \HNN_f(\calG(v),\partial_1(\calG(e)),t)
\end{align*}
and similar to above $(\calG,\Gamma)$ is proper if and only if the HNN-extension is proper.

The fundamental pro-$p$ group of a finite graph of pro-$p$ groups $(\calG,\Gamma)$ can be obtained by choosing a maximal subtree $\mathrm T$ of $\Gamma$, taking the iterated amalgamated free pro-$p$ products of vertex groups along the edge groups and taking \HNN-extensions for each remaining loop.
Hence, Propositions~\ref{prop:amalg} and \ref{prop:HNN} applied successively prove the following theorem.

\begin{thm}\label{thm:fundamentalgroupgeneral}
Let $(\calG, \Gamma)$ be a proper finite graph of pro-$p$ groups. Let $G = \Pi_1(\calG, \Gamma)$, let $T \subset \Gamma$ be a maximal tree, and denote $G_T = \Pi_1(\calG, T)$. Suppose that:
\begin{enumerate}[label=(\arabic*)]
\item \label{it:thm fundamental groups Gv trivial radical} For all $v \in V(\Gamma)$, $\calR_p(\calG(v)) = 0$.

\item \label{it:thm fundamental groups restriction maps surjective} For all $e \in E(\mathrm T)$, the restriction maps $\partial_{i,e}^{(1)}$ are surjective for $i = 0,1$.

\item \label{it:thm fundamental groups diagram commuatative} For all $e \not\in E(\mathrm T)$, the following square commutes:
\begin{equation*}
    \begin{tikzcd}[column sep=small]
        \rmH^1(G_T,\FF_p)\arrow[r]\arrow[d]&\rmH^1(\calG(d_0(e)),\FF_p)\arrow[d]\\
        \rmH^1(\calG(d_1(e)),\FF_p)\arrow[r]&\rmH^1(\calG(e),\FF_p)
    \end{tikzcd}
\end{equation*}
and the induced map $\rmH^1(G_T, \FF_p) \to \rmH^1(\calG(e), \FF_p)$ is surjective.
 \end{enumerate}
Then $\calR_p(G) = 0$.
\end{thm}

\begin{proof}
Because $T$ is a tree, $G_T$ splits as an iterated amalgamated product
$$G_1 \amalg_{H_1} G_2 \amalg_{H_2} \cdots G_m \amalg_{H_{m-1}} G_m$$
of finitely many factors where each $\calR_p(G_j) = 0$ for $j=1,...,m$. At each amalgamation step the surjectivity of the restriction map is preserved, then applying Proposition \ref{prop:amalg} inductively implies $\calR_p(G_T)=0$.

Let $e_1,\dots,e_n$ be the edges of $\Gamma \smallsetminus E(\mathrm T)$. The group $G = \Pi_1(\calG,\Gamma)$ is obtained iteratively by setting $G_0 = G_T$ and $G_k = \HNN(G_{k-1}, \calG(e_k), \partial_0^k, \partial_1^k)$, where $\partial_i^k$ is the composition of the edge monomorphism with the inclusion into $G_T$. We proceed by induction on $k$ to prove that $\calR_p(G_k) = 0$ and $\res^1_{G_k,G_T}$ is surjective. The base case $k = 0$ holds trivially.

Assume the claim holds for every natural $< k$. Let $A = \calG(e_k)$. The commutative diagram in \ref{it:thm fundamental groups diagram commuatative} ensures that the maps induced by the inclusions $\partial_0^k$ and $\partial_1^k$ coincide on $\rmH^1(G_{k-1},\FF_p)$. Thus, their difference in the Mayer--Vietoris sequence for $G_k$ is trivial, forcing $\rmH^1(G_k,\FF_p) \to \rmH^1(G_{k-1},\FF_p)$ to be surjective by exactness. 

Consequently, $\res^1_{G_k,G_T}$ is surjective. Moreover, the surjectivity of the maps in \ref{it:thm fundamental groups diagram commuatative} and the inductive hypothesis ensure that the maps $(\partial_i^k)^{(1)}$ are surjective. Since $\calR_p(G_{k-1}) = 0$, by Proposition \ref{prop:HNN}, we conclude $\calR_p(G_k) = 0$.
\end{proof}
\begin{ex}
Consider the following graph $\Gamma$:
$$\begin{tikzpicture}[auto]
\tikzset{every loop/.style={}}
\node[circle,draw=black, fill=black, inner sep=0pt, minimum size=1.5mm] at (0,0) {};
\node[circle,draw=black, fill=black, inner sep=0pt, minimum size=1.5mm] at (2,0) {};
\draw node[below=0.2cm] {$v_1$} (0,0) -- node[above] {$e_1$} (2,0) node[below = 0.1cm] {$v_2$};
\path (2,0) edge[in=330, out=30, loop,distance=1.5cm] node[right] {$e_2$} (2,0) node[above] {};
\end{tikzpicture}$$
The associated vertex groups are $\calG(v_i)\cong \ZZ_p^2$. We denote by $a_i$ and $b_i$ generators of $\calG(v_i)$. We choose $\calG(e_i)=\ZZ_p$ and denote by $c_i$ a generator of $\calG(e_i)$.

Now the edge maps are defined by
\begin{align*}
    \partial_{0,e_1}(c_1) &= a_1, &\partial_{1,e_1}(c_1) &= a_2,\\
    \partial_{0,e_2}(c_2) &= a_2 &\partial_{1,e_2}(c_2) &= a_2 b_2^p
\end{align*}
By identifying $a_1$ and $a_2$ within the presentation of the fundamental group $G$, it is not hard to check that this finite graph of pro-$p$ groups is proper. One readily verifies \ref{it:thm fundamental groups Gv trivial radical} and \ref{it:thm fundamental groups restriction maps surjective}. Since the images of the loop edge under the respective embeddings differ by the Frattini subgroups, hence, the induced maps on the first cohomology groups are identical and therefore also \ref{it:thm fundamental groups diagram commuatative} holds. Thus, $R_p(G) = 0$ by Theorem~\ref{thm:fundamentalgroupgeneral}.
\end{ex}

\section{The \texorpdfstring{$\FF_p$}{}-cup radical of \texorpdfstring{$p$}{}-RAAGs and \texorpdfstring{$\Delta$}{}-RAAGs}
\label{sec:RAAGs and Delta-RAAGs}
In this section, we present two classes of pro-$p$ groups that generalize in two different manners that of right-angled Artin pro-$p$ groups. 
The class of \textit{generalized $p$-RAAGs} appeared in a joint work \cite{quadrelliQuadraticCoho} by the last named author together with I.~Snopce and M.~Vannacci; while the second class, that of \textit{$\Delta$-RAAGs}, has been recently introduced by O.~Hamza, C.~Maire, J.~Mina\v c and N.D.~T\^an in \cite{deltaRAAGs} to study the maximal pro-$2$ quotients of the absolute Galois groups of formally real pythagorean fields; indeed, their construction allows one to introduce torsion inside the RAAG. 

We show that all these groups satisfy the $p$-Kijima--Nishi property, regardless of their realizability as maximal pro-$p$ quotients of absolute Galois groups. 

\subsection{\texorpdfstring{$p$}{}-RAAGs}\label{ssec:pRAAGs}
Let $\Gamma=(V,E)$ be a finite combinatorial directed graph, i.e., $V$ is a finite set of vertices, and $E\subseteq V\times V$ is a set of edges such that, if $(v,w)\in E$, then $(w,v)\notin E$.
A $p$-\textit{graph} is a pair $(\Gamma,f)$, where $f$ is a map $f:E\to p\mathbb Z_p\times p\mathbb Z_p$, if $p$ is odd, and $f:E\to 4\mathbb Z_2\times 4\mathbb Z_2$ if $p=2$. The associated generalized $p$-RAAG is the pro-$p$ group with presentation 
\[G_{\Gamma,f}=\pres{v\in V}{ [v,w]=v^{\lambda_1}w^{\lambda_2}:\ e=(v,w)\in E,\ (\lambda_1,\lambda_2)=f(e)}.\]
For the trivial map $f:e\in E\mapsto (0,0)$, the $\mathbb F_p$-cohomology ring of the pro-$p$ RAAG $G_\Gamma=G_{\Gamma,0}$ is known to be the exterior Stanley-Reisner ring (cf. \cite{lorensen,bartholdi})
\[{\rm H}^\bullet(G_\Gamma,\mathbb F_p)=\frac{\Lambda_{\mathbb F_p}(v^\ast:\ v\in V)}{(v^\ast\wedge w^\ast:\ (v,w),(w,v)\notin E)}\]
where $(v^\ast)_{v\in V}$ denotes the dual basis of the basis $(v\Phi(G_\Gamma))_{v\in V}$ of $G_\Gamma/\Phi(G_\Gamma)$.
On the contrary, for a general labeling $f$, the whole $\mathbb F_p$-cohomology of $G_{\Gamma,f}$ is unknown in general, but one can easily describe the cup product 
$$\rmH^1(G_{\Gamma,f},\mathbb F_p)\times \rmH^1(G_{\Gamma,f},\mathbb F_p)\longrightarrow \rmH^2(G_{\Gamma,f},\mathbb F_p)$$ using Proposition~\ref{prop:cup product and shape of relations}.
One has $$\rmH^1(G_{\Gamma,f},\FF_p)=\Hom_{\rm cts}(G_{\Gamma,f},\FF_p)={\rm Span}(v^\ast\mid v\in V),$$ where $v^\ast(w)=\delta_{vw}$ for $v,w\in V$, and there is a natural map $\bfH(G_\Gamma,\FF_p)\to \bfH(G_{\Gamma,f},\FF_p)$ which is an isomorphism in degrees $0$, $1$, and $2$.

\begin{lem}
\label{lem:radical of RAAG}
Let $(\Gamma=(V,E),f)$ be a $p$-graph. The $\FF_p$-cup radical of $G_{\Gamma,f}$ is the linear span of the dual basis elements of isolated vertices, i.e., \[\mathcal R_p(G_{\Gamma,f})=\operatorname{Span}_{\mathbb F_p}\set{v^\ast\mid v\in V}{ (w,v),(v,w)\notin E\ \forall w\in V}.\]
\end{lem}
\begin{proof} 
    Since the cup product in degree one does not depend on the labeling $f$, it is enough to prove the claim for $f=0$. 
    If $v\in V$ is an isolated vertex, then $v^\ast\smallsmile \rmH^1(G_{\Gamma,f},\F_p)=0$ in $\bfH(G_\Gamma,\mathbb F_p)$. 
    
    Now, for $V=\{v_1,\ldots,v_d\}$, put 
    $$\alpha=a_1 v_1^\ast+\ldots+a_dv_d^*\in {\rm H}^1(G_\Gamma,\mathbb F_p),$$ and suppose that $a_i\neq 0$ for a non-isolated vertex $v_i\in V$. 
    Let $\{v_i,v_j\}\in E$ be an edge containing $v_i$, and consider the cup product 
    $$\alpha\smallsmile v_j^\ast=a_i (v_i^\ast\smallsmile v_j^\ast)+\sum_{h\neq i}a_h (v_h^\ast\smallsmile v_j^\ast).$$
    Since the set $\set{v_h^\ast\smallsmile v_{h'}^\ast}{h<h',\{v_h,v_{h'}\}\in E}$ is a basis of ${\rm H}^2(G_\Gamma,\mathbb F_p)$, we have proved that $\alpha\smallsmile v_j^\ast\neq 0$, that is $\alpha\notin \mathcal R_p(G_\Gamma)$.
\end{proof}

\begin{cor}
    \label{cor:RAAGs have KN}
    Let $(\Gamma=(V,E),f)$ be a $p$-graph. Then, the generalized pro-$p$ RAAG $G_{\Gamma,f}$ has the $p$-Kijima--Nishi property.
\end{cor}
\begin{proof}
    Let $I$ be the set of isolated vertices of $\Gamma$, and let $\Gamma'$ be the induced subgraph of $\Gamma$ spanned by $V\smallsetminus I$. If $F$ is the free pro-$p$ group on $I$, then $G_{\Gamma,f}$ splits as a free pro-$p$ product $F\amalg G_{\Gamma',f'}$
    (here $f'$ denotes the restriction of $f$ to $\Gamma'$), concluding the proof in the light of Proposition \ref{prop:KNFreeProduct}.
\end{proof}

\subsection{\texorpdfstring{$\Delta$}{}-RAAGs}\label{ssec:deltaRAAGs}
 Let $\Gamma=(V,E)$ be a finite graph, and let $G_\Gamma$ be its associated pro-$p$ RAAG. Denote by $\Delta$ the group of order two, and let $x_0$ be its generator. Let ${\bf z}:V\to G_\Gamma$ be a set-theoretic function such that the induced map $a_{\bf z}:G_\Gamma\to G_\Gamma$ defined by $v\mapsto {\bf z}_v v^{-1}{\bf z}_v^{-1}$ is an automorphism of $G_\Gamma$ of order two. 

These data define a $\Delta$-action $\delta_{\bf z}:\Delta\to \operatorname{Aut}(G_\Gamma)$ defined by $\delta_{\bf z}(x_0)=a_{\bf z}$, and hence a split extension $G_{\Gamma,\bf z}\coloneq G_\Gamma\rtimes_{\delta_{\bf z}}\Delta$, called the $\Delta$-RAAG associated with $\Gamma$ and $\bf z$ (see \cite{deltaRAAGs}).

\begin{prop}
\label{prop:D-RAAGs have KN}
    If $\Gamma$ is a non-empty finite graph, and $G_{\Gamma,\bf z}$ is a $\Delta$-RAAG defined on $\Gamma$, then $\calR_2(G_{\Gamma,\bf z})=0$, and hence $G_{\Gamma,\bf z}$ satisfies the $2$-Kijima--Nishi property.
\end{prop}
\begin{proof}
{%\color{julian} 
Since $\Delta$ acts by conjugation with the elements $\mathbf{z}_v$ on the generators $v$ of $G_\Gamma$ we see that the induced action on $\rmH^1(G_\Gamma,\FF_2)$ is trivial and hence $G_{\Gamma,\mathbf{z}}$ splits as a $2$-almost direct product and, by Proposition~\ref{prop:almost direct product}, its $\FF_2$-cup radical is trivial.}
\end{proof}
As a consequence, we deduce case \ref{it:Main Theorem Fields Pythagorean} from Theorem~\ref{mainthm:List of fields}: if $\KK$ is a formally real Pythagorean field of finite type --- i.e., if $-1$ is not a square, the sum of two squares is a square, and $\KK^\times/\KK^{\times 2}\simeq \rmH^1(\KK,\F_2)$ is a finite group ---, then its maximal pro-$2$ Galois group is a $\Delta$-RAAG (cf. \cite[Thm. 2.1]{deltaRAAGs}), and hence $\KK$ satisfies the $2$-Kijima--Nishi property.

\section{A generalization in the absence of roots of unity}
\label{sec:No roots of unity}
If a field $\KK$ does not contain a primitive root of unity, then the definition of the radical already poses problems. It is not hard to see that, for example, $\calR_p(G_{\QQ_\ell})=\FF_p$ when $p\nmid \ell-1$. In this section we adapt our definition of the $\FF_p$-cup radical to this situation and summarize generalizations of some of the results given in Section~\ref{sec:Arithmetic and geometric fields} to fields that do not contain a primitive $\pth$-root of unity. 

We first generalize our definition for the radical for \emph{$p$-oriented profinite groups}. These are pairs $(G,\theta)$, where $G$ is a profinite group and $\theta:G\to \ZZ_p^\times$ a continuous homomorphism. For $k\in \ZZ$, we define $\FF_p(k)$ to be the discrete $G$-module with $\sigma(x)\coloneq \theta(\sigma)^k\cdot x$. More generally, if $M$ is a discrete $p$-torsion $G$-module, one defines the $k^{\rm th}$-Tate twist of $M$ to be $M(k)\coloneq M\otimes_{\FF_p} \FF_p(k)$.

We now define
\begin{align*}
    \calR_p(G,\theta)\coloneq \set{\alpha\in \rmH^1(G,\FF_p)}{\alpha \smallsmile\beta=0\text{ for all }\beta \in \rmH^1(G,\FF_p(1))}.
\end{align*}
If $G=G_{\KK}$ for a field $\KK$ of characteristic $\neq p$, then there is a natural $p$-orientation $\theta_{\KK}$ given by the action of $G$ on $\mu_{p^\infty}(\KK^{\rm sep})\cong \QQ_p/\ZZ_p$, as $\Aut(\QQ_p/\ZZ_p)\cong \ZZ_p^\times$. Then the cup product between elements of $\rmH^1(G,\F_p)$ and $\rmH^1(G,\mu_p)$ can be identified with the map
\begin{align*}
    \rmH^1(G,\FF_p)\otimes \rmH^1(G,\mu_p)\to {}_p\Br(\KK),\quad \chi\otimes \kappa(b)\mapsto [(\chi,b)],
\end{align*}
where $(\chi,b)$ is the cyclic algebra associated with $b\in\KK^\times$ and $\chi$ (see \cite[Const.~2.5.1 and Prop.~4.7.1]{GilleSzamuely}). Thus, this generalizes Remark~\ref{rem:Brauer group and cup product}.

The next lemma is clear from the definitions and shows that $\calR_p(G,\theta)$ essentially only depends on the reduction of $\theta$ modulo $p$.
\begin{lem}
    Let $\theta$ and $\theta'$ be two $p$-orientations on a profinite group $G$. If $\theta(\sigma)\equiv \theta'(\sigma)\pmod p$ for every $\sigma \in G$, then $\calR_p(G,\theta)=\calR_p(G,\theta')$.

    In particular, if $\im\theta\subseteq 1+p\ZZ_p$, then $\calR_p(G,\theta)=\calR_p(G)$.
\end{lem}
We generalize Corollary~\ref{cor:KNsemidirect} to this more general context. We give a full proof of the statement, as it highlights the techniques quite well.
\begin{prop}[Cyclotomic semidirect products]
    Let $(G_0,\theta)$ be a $p$-oriented profinite group. Define $G\coloneq \ZZ_p\rtimes_{\theta} G_0$ and let $\theta':G\to \ZZ_p^\times$ be the composition of $\theta$ with the projection onto the second factor. 

    Then $\calR_p(G,\theta')=0$. In particular, for any $p$-henselian field $\KK$, whose residue field is of characteristic $\neq p$ and not $p$-closed one has $\calR_p(G_{\KK},\theta_{\KK})=0$.
\end{prop}
\begin{proof}
    We can without loss of generality assume that $\theta\not\equiv 1\pmod p$, since otherwise the semidirect product is $p$-almost direct and hence the statement follows by Proposition~\ref{prop:almost direct product}.

    Let $k\in \ZZ$, then we have a spectral sequence 
    \begin{align*}
        E^{s,t}_2(k)\coloneq \rmH^s(G_0,\rmH^t(\ZZ_p,\FF_p(k)))\Longrightarrow \rmH^{s+t}(G,\FF_p(k))
    \end{align*}
    Note that $\rmH^0(\ZZ_p,\FF_p(k))\cong \FF_p(k)$ and $\rmH^1(\ZZ_p,\FF_p(k))\cong \FF_p(k-1)$ as $G_0$-modules. By \cite[Prop. 2.4.5]{NSW2008}, the spectral sequence degenerates, and we have for every $n\in \NN$ exact sequences
    \begin{align*}
        0\to \rmH^n(G_0,\FF_p(k))\to \rmH^n(G,\FF_p(k))\overset{\partial^n(k)}\to \rmH^{n-1}(G_0,\FF_p(k-1))\to 0.
    \end{align*}
    From this, we immediately see that $\rmH^1(G_0,\FF_p)\cong \rmH^1(G,\FF_p)$. Note that for $k,k'\in \ZZ$ there is a canonical product $E^{*,*}_2(k)\otimes E_2^{*,*}(k')\to E^{*,*
    }_2(k+k')$ coming from the isomorphism $\FF_p(k+k')\cong \FF_p(k)\otimes \FF_p(k')$.

    We note that the following diagram commutes:
    \begin{equation*}
        \begin{tikzcd}[column sep=small]
            \rmH^1(G,\FF_p)\arrow[r,phantom,"\times"]&\rmH^1(G,\FF_p(1))\arrow[d,two heads, "\partial^1(1)"] \arrow[r,"\smallsmile"] &\rmH^2(G,\FF_p(1))\arrow[d,two heads,"\partial^2(1)"]\\
            \rmH^1(G_0,\FF_p)\arrow[u,equal]\arrow[r,phantom,"\times"]&\rmH^0(G_0,\FF_p)\arrow[r,"\smallsmile"]&\rmH^1(G_0,\FF_p)
        \end{tikzcd}
    \end{equation*}
    By the identification $\rmH^0(G_0,\FF_p)\cong \FF_p$ the cup product in the lower row is just multiplication by scalars.
    
    Given $\alpha\in \calR_p(G,\theta')$, we let $\alpha_0\in \rmH^1(G_0,\FF_p)$ be the unique element such that $\smash{\inf^1}(\alpha_0)=\alpha$. Now choose $\beta\in \rmH^1(G,\FF_p(1))$ such that $\lambda\coloneq \partial^1(1)(\beta)\neq 0$. Then 
    \begin{align*}
        0=\partial^2(1)(\alpha \smallsmile\beta)=\alpha_0\smallsmile \partial^1(1)(\beta)=\lambda \alpha_0\in \rmH^1(G_0,\FF_p)
    \end{align*}
    and thus $\alpha_0=0$. This shows the first statement. The second one follows from the first one by writing $G_\KK$ as semidirect product of the absolute Galois group of the residue field with a procyclic group (see \cite[Exercises II \S4]{serre:galc}), where the action is given by a variant of $\theta_\KK$. The same arguments as above apply in this case.
\end{proof}
We recover several of the statements of Section~\ref{sec:Arithmetic and geometric fields} in this more general setting, leading us to believe that this generalized definition of the radical is indeed the right notion in the case where $\KK$ does not contain a primitive $\pth$-root of unity.
\begin{thm}
\label{thm:general fields}
    Let $\KK$ be either of the following fields and $p$ a prime different from the characteristic of $\KK$, then $\calR_p(\KK,\theta_{\KK})=0$:
    \begin{enumerate}[label=(\roman*)]
        \item \label{it:general local field} local field;
        \item \label{it:general global field} global field;
        \item \label{it:general function field} a purely transcendental extension of a field $\mathbb{k}$, for which $G_{\mathbb{k}}(p)$ is infinite.
    \end{enumerate}
\end{thm}
\begin{proof}
    The case \ref{it:general local field} follows immediately from local Tate duality and the proofs of \ref{it:general global field} and \ref{it:general function field} follow verbatim the ones of Theorem~\ref{thm:Gloabal field trivial radical} and Theorem~\ref{thm: rational function field trivial radical}.
\end{proof}